\documentclass[12pt, reqno]{amsart}
\usepackage{amsmath}
\usepackage{amssymb}
\usepackage{amsfonts}
\usepackage{amsthm}
\usepackage{amscd}
\usepackage{relsize}
\usepackage{setspace}
\usepackage{geometry}
\usepackage{enumerate}
\usepackage{url}
\usepackage{xspace}
\usepackage{mathrsfs}
\usepackage{dsfont}
\usepackage{lmodern}
\usepackage{xcolor}
\usepackage[utf8]{inputenc}
\usepackage{mathtools}

\usepackage{hyperref}

\usepackage[textsize=footnotesize,textwidth=15ex]{todonotes}

\newtheorem{thmintro}{Theorem}

\newtheorem{corintro}[thmintro]{Corollary}

\newtheorem{theorem}{Theorem}[section]
\newtheorem{lemma}[theorem]{Lemma}
\newtheorem{prop}[theorem]{Proposition}
\newtheorem{corollary}[theorem]{Corollary}

\theoremstyle{definition}
\newtheorem{remark}[theorem]{Remark}
\newtheorem{example}[theorem]{Example}

\numberwithin{equation}{section}

\newcommand{\BBB}{\mathcal{B}}
\newcommand{\UU}{\mathcal{U}}

\newcommand{\FF}{\mathbb{F}}
\newcommand{\KK}{\mathbb{K}}
\newcommand{\NN}{\mathbb{N}}
\newcommand{\QQ}{\mathbb{Q}}
\newcommand{\ZZ}{\mathbb{Z}}

\newcommand{\g}{\mathfrak{g}}
\newcommand{\h}{\mathfrak{h}}
\newcommand{\n}{\mathfrak{n}}

\newcommand{\G}{\mathfrak{G}}

\newcommand{\U}{\mathfrak{U}}

\newcommand{\llangle}{\langle\!\langle}
\newcommand{\rrangle}{\rangle\!\rangle}

\DeclareMathOperator{\ad}{ad}

\DeclareMathOperator{\charact}{char}
\DeclareMathOperator{\height}{ht}
\DeclareMathOperator{\modulo}{mod}
\DeclareMathOperator{\SL}{SL}
\DeclareMathOperator{\Ker}{Ker}

\newcommand{\inv}{^{-1}}

\newcommand{\co}{\colon\thinspace}

\DeclareMathOperator{\Lie}{L}

\begin{document}

\title[Amalgams of rational unipotent groups]{Amalgams of rational unipotent groups\\ and residual nilpotence}

\author{Pierre-Emmanuel \textsc{Caprace}}
\author{Timoth\'{e}e \textsc{Marquis}}

\thanks{TM is a F.R.S.-FNRS Research associate; PEC and TM are supported in part by the FWO and the F.R.S.-FNRS under the EOS programme (project ID 40007542).}

\begin{abstract}
We provide sufficient conditions for a free amalgamated product of torsionfree nilpotent groups to be residually nilpotent. We also characterise the residual nilpotence of certain higher-dimensional  amalgams of unipotent groups over the rationals (known as KMS groups)  in terms of their defining Cartan matrix. As an application, we give a normal form for the elements of a minimal Kac--Moody group over the rationals. 
\end{abstract}

\maketitle

\begin{flushright}
\begin{minipage}[t]{0.55\linewidth}\itshape\small
\dots le monde naturel pénètre dans le spirituel, lui sert de pâture et concourt ainsi à opérer cet amalgame indéfinissable\dots

\vspace{2mm}

\hfill\upshape \textemdash C. Baudelaire,  \emph{\OE uvres complètes}, 1931.
\end{minipage}
\end{flushright}

\section{Introduction}

\subsection{Residual nilpotence of free amalgamated products}

Following standard terminology introduced by K.~Gr\"unberg \cite{Grunberg57}, given a group property $\mathscr P$, we say that a group $G$ is \textbf{residually $\mathscr P$} if every non-trivial element of $G$ has a non-trivial image in some quotient of $G$ that satisfies the property $\mathscr P$. 
The study of residual properties of graphs of groups has a long and rich history,  originating from Magnus' theorem that free groups are residually torsionfree nilpotent \cite{Magnus35}. It has played an influential role in recent developments related to $3$-manifold groups, see \cite{AFW}. Fairly definitive results have been established characterising  the residual $p$-finitess of graphs of groups (see \cite{Wilkes19}), drawing inspiration from Higman's foundational paper on \textit{amalgams of $p$-groups} \cite{Higman64}. 

Our understanding of residual nilpotence is not nearly as complete, probably    partly by virtue of  the fact that, as opposed to finite $p$-groups, the class of nilpotent groups is not stable under group extensions. Baumslag, who devoted a large number of papers to this subject throughout his career,  wrote in 2010 that ``\textit{proving that a given group is residually torsionfree nilpotent is usually hard}'' (see \cite{Baumslag10}). While we feel that Baumslag's assertion remains true today, it is worth mentioning that an important breakthrough has recently been accomplished by Jaikin-Zapirain (see \cite{Jaikin22}).  Concerning the residual nilpotence of graphs of groups, let us merely recall here that the free product  $C_p * C_q$ of cyclic groups of distinct prime order~$p$ and~$q$ is not residually nilpotent\footnote{Indeed, finitely generated nilpotent groups are residually finite by \cite{Hirsch}. Hence a finitely generated residually nilpotent group is residually \{finite nilpotent\}, but the largest finite nilpotent quotient of $C_p * C_q$ is $C_p \times C_q$.}, showing that residual nilpotence is not stable under forming free products. By contrast, Mal'cev \cite{Mal49} showed that the free product of residually \textit{torsionfree nilpotent} groups is itself residually torsionfree nilpotent. On the other hand, a free \textit{amalgamated} product of finitely generated torsionfree nilpotent groups may fail to be residually nilpotent, or even residually finite, as shown by a remarkable example due to Baumslag \cite{Baumslag62}. To the best of our knowledge, necessary and sufficient conditions on a free amalgamated product $A*_C B$ of torsionfree nilpotent groups to be residually nilpotent (formulated intrinsically in terms of the underlying edge of groups, in the same vein as in \cite{Higman64}), are currently not known. The following result is a contribution to this problem.

\begin{thmintro}	\label{thmintro:free-amalgamated-prod}
The free amalgamated product $A*_C B$ is residually nilpotent as soon as the following conditions hold:
\begin{enumerate}
\item $A$ and $B$ are torsionfree nilpotent;
\item $C$ is a retract\footnote{A subgroup $C$ of a group $H$ is a \textbf{retract} if there is a homomorphism $\varphi \colon H \to C$ whose restriction to $C$ is the identity. In other words, $H$ splits as the semi-direct product $H \cong \Ker(\varphi) \rtimes C$.}
 of both $A$ and $B$;
\item $C$ is abelian.
\end{enumerate}
\end{thmintro}

The free product   $C_p * C_q$ mentioned above shows that  the torsionfreeness condition in (1) cannot be discarded. Similarly, Baumslag's example from \cite{Baumslag62} shows that the condition (2) cannot be removed either.   It is  likely that the hypothesis (3) is redundant: indeed, this could follow by sharpening the commutator calculus involved in the proof of Theorem~\ref{thmintro:free-amalgamated-prod}. We shall not pursue that goal here, as our main purpose is to highlight the intriguing phenomena that one observes when passing from edges of groups to higher dimensional amalgams.

Before doing so, let us   mention that we will also provide sufficient conditions for $A*_C B$ to be residually torsionfree nilpotent, see  Remark~\ref{rmk:res-nilp-IFF-tf-res-nilp} below. As a consequence, we will see that if $U_1 \cong U_2 \cong U_3$ are three copies of the additive group of $\QQ$, then the group
$$U_1 * U_2 * U_3/ \llangle [U_1, [U_1, U_2]], [U_2, [U_2, U_1]], [U_2, [U_2, U_3]],  [U_3, [U_3, U_2]]\rrangle$$ is residually torsionfree nilpotent. Notice that the latter group is a free amalgamated product of two copies of the Heisenberg group over $\QQ$, since the group
$U_1 * U_2 / \llangle [U_1, [U_1, U_2]], [U_2, [U_2, U_1]]\rrangle$ is isomorphic to the Heisenberg group over $\QQ$ (see Example~\ref{example:simplerpresentationrank2}).

\subsection{A triangle of groups with rational Heisenberg vertex groups}

Let us now consider the group 
$$\mathcal U_{\widetilde A_2}(\QQ) := U_1 * U_2 * U_3/ \llangle [U_i, [U_i,  U_j]], \ i \neq j \in \{1, 2, 3\}\rrangle$$  
where, as before, the groups $U_1, U_2, U_3$ are copies of $\QQ$. 
The group  $\mathcal U_{\widetilde A_2}(\QQ)$  may be viewed as the fundamental group of a triangle of groups whose edge groups are $U_1$, $U_2$, $U_3$, and whose vertex groups are three copies of the Heisenberg group over $\QQ$ (see Example~\ref{example:simplerpresentationrank2}).

For each $i \in \{1, 2, 3\}$, let us fix an isomorphism $x_i \colon \QQ \to U_i$. Let $\QQ[t]$ be the polynomial ring in the indeterminate $t$. One observes easily that the assignments 
$$
x_1(a)\mapsto \begin{psmallmatrix}1&a&0\\ 0&1&0\\ 0&0&1\end{psmallmatrix}, \quad x_2(a)\mapsto \begin{psmallmatrix}1&0&0\\ 0&1&a\\ 0&0&1\end{psmallmatrix}\quad\textrm{and}\quad x_3(a)\mapsto \begin{psmallmatrix}1&0&0\\ 0&1&0\\ at&0&1\end{psmallmatrix}\quad\textrm{for all $a\in\QQ$},
$$
extend to  a group homomorphism $\varphi \colon \mathcal U_{\widetilde A_2}(\QQ) \to \SL_3(\QQ[t])$. The image of $\varphi$ is the subgroup of $\SL_3(\QQ[t])$ consisting of those matrices that are upper triangular modulo~$t$. The following result characterises the kernel of $\varphi$ as the \textbf{nilpotent residual} of $\mathcal U_{\widetilde A_2}(\QQ)$, i.e. the intersection of all terms of its lower central series. 

\begin{thmintro}\label{thmintro:A-2-tilde}
The kernel of $\varphi \colon \mathcal U_{\widetilde A_2}(\QQ) \to \SL_3(\QQ[t])$ coincides with $\gamma_\infty(\mathcal U_{\widetilde A_2}(\QQ))$. 

Moreover,  the subgroup of $ \mathcal U_{\widetilde A_2}(\QQ) $ generated by $[U_1, U_2] \cong \QQ$ and $[U_2, U_3] \cong \QQ$ is isomorphic to the free product $\QQ *  \QQ$, whereas its image under $\varphi$ is isomorphic to the Heisenberg group over $\QQ$. 

In particular, $\gamma_\infty(\mathcal U_{\widetilde A_2}(\QQ))$ contains $[[U_1, U_2], [[U_1, U_2], [U_2, U_3]]] \neq \{e\}$. Hence the group $ \mathcal U_{\widetilde A_2}(\QQ)$ is not residually nilpotent. 
\end{thmintro}

The proof has two parts. The first part, namely the characterisation of $\Ker(\varphi)$ as the nilpotent residual, relies on Lie ring techniques. In order to be more precise, let us fix some notation. Given a group $H$, we write $\gamma_1(H):=H$ and $\gamma_{n+1}(H):=[H,\gamma_n(H)]$ ($n\geq 1$) for its lower central series. The space $\Lie(H):=\bigoplus_{n\geq 1}\gamma_{n}(H)/\gamma_{n+1}(H)$ is then naturally an $\NN$-graded \textbf{Lie ring}, with Lie bracket defined by
$$[g\gamma_{m+1}(H),h\gamma_{n+1}(H)]:=[g,h]\gamma_{m+n+1}(H)\quad\textrm{for all $g\in\gamma_m(H)$ and $h\in\gamma_n(H)$,}$$ 
where $[g,h]:=g\inv h\inv gh$ is the group commutator bracket.
We shall show that the Lie ring  $\Lie(\mathcal U_{\widetilde A_2}(\QQ))$ is a Lie algebra over $\QQ$, which happens to be  isomorphic to the positive part of the affine Kac--Moody algebra of type $\widetilde A_2$ over $\QQ$. 

For the second part, namely the description of the subgroup generated by $[U_1, U_2] \cong \QQ$ and $[U_2, U_3] \cong \QQ$, we use  basic tools from CAT(0) geometry, that are available in this context through the natural action of  $\mathcal U_{\widetilde A_2}(\QQ)$ on a CAT(0) triangle complex (see the proof of Theorem~\ref{thm:freeproductKMS}).

Both parts can actually be developed in a more  general context that we now proceed to describe. 

\subsection{Amalgams of rational unipotent groups}

The group $\mathcal U_{\widetilde A_2}(\QQ)$ is a special instance belonging to a family of  amalgams of maximal unipotent subgroups of rank $2$ Chevalley groups over an arbitrary field $\KK$, known as \emph{Kac--Moody--Steinberg} (KMS) \emph{groups}.

A KMS group $\UU_A(\KK)$ over a field $\KK$ is associated with a \emph{generalised Cartan matrix} (GCM) $A=(a_{ij})_{i,j\in I}$ (see \cite[Chapter~1]{Kac}) that is \emph{$2$-spherical}, that is, such that each submatrix $A_{ij}=\begin{psmallmatrix}2&a_{ij}\\ a_{ij}&2\end{psmallmatrix}$ with $i\neq j$ is a Cartan matrix (hence of one of the types $A_1\times A_1$, $A_2$, $B_2$ or $G_2$). It is an amalgamated product over all pairs $\{i,j\}$ of the unipotent radical $U_{ij}$ of the standard Borel subgroup of the Chevalley group of type $A_{ij}$ over $\KK$ (see \S\ref{section:prelim} for a precise definition). A specific KMS group over $\KK=\FF_p$ appeared in \cite[Prop.~7.4]{EJ10} and \cite[Thm.~12.1]{Ershov_GS} as an example of a Golod--Shafarevich group with property (T). More recently, KMS groups over finite fields were also shown to be promising sources of high-dimensional expanders (see \cite{KMSHDX}).

Let $\g_{A,\QQ}$ denote the (derived) \emph{Kac--Moody algebra} of type $A$ over $\QQ$ (defined by the same Serre presentation over $\QQ$ as semisimple Lie algebras), with triangular decomposition $\g_{A,\QQ}=\n^-_{A,\QQ}\oplus\h_{\QQ}\oplus\n^+_{A,\QQ}$ and corresponding set of roots $\Delta$, as in \cite[Chapter~1]{Kac}. Let also $\G_A(\QQ)$ be the \emph{minimal Kac--Moody group} of type $A$ over $\QQ$, and let $\U^+_A(\QQ)$ be the subgroup of $\G_A(\QQ)$ generated by all \emph{root groups} $\U_{\alpha}(\QQ)$ associated to positive real roots $\alpha\in\Delta$, as in \cite{Tits87} (see also \cite[Definition~7.47]{KMGbook}). Finally, let $\U^{ma+}_A(\QQ)$ denote the completion of $\U^+_A(\QQ)$ in the \emph{maximal Kac--Moody group} $\G^{pma}_A(\QQ)$ (see \cite[\S8.5.1]{KMGbook}). Then there are canonical group morphisms (see \S\ref{section:prelim})
$$\UU_A(\QQ)\to \U^+_A(\QQ)\to\U^{ma+}_A(\QQ).$$
For instance, in the special case where the matrix $A=\begin{psmallmatrix}2&-1&-1\\ -1&2&-1\\-1&-1&2\end{psmallmatrix}$ is of affine type $\widetilde{A}_2$, the map 
$\UU_A(\QQ)\to \U^+_A(\QQ)$ is the homomorphism $\varphi$ appearing above,  and the groups $\U^+_A(\QQ)$ and $\U^{ma+}_A(\QQ)$ respectively coincide with the subgroup of matrices in $\SL_3(\QQ[t])$ and $\SL_3(\QQ[\![t]\!])$ that are upper triangular modulo $t$.

Our main theorem asserts that the Lie rings associated to $\UU_A(\QQ)$ and $\U^+_A(\QQ)$ admit natural $\QQ$-Lie algebra structures, and are canonically isomorphic to $\n^+_{A,\QQ}$ as $\QQ$-Lie algebras.

\begin{thmintro}\label{thmintro:main}
Let $A$ be a $2$-spherical GCM. Then the canonical group morphisms $\UU_A(\QQ)\to \U^+_A(\QQ)\to\U^{ma+}_A(\QQ)$ induce isomorphisms of $\QQ$-Lie algebras
$$\Lie(\UU_A(\QQ))\stackrel{\cong}{\to}\Lie(\U^+_A(\QQ))\stackrel{\cong}{\to} \n^+_{A,\QQ}.$$
\end{thmintro}

Assume for the rest of this introduction that $A$ is a $2$-spherical GCM.

As a first immediate consequence of Theorem~\ref{thmintro:main}, we can compute the pro-nilpotent completions of $\UU_A(\QQ)$ and $\U^+_A(\QQ)$. For a group $H$, denote by $\widehat{H}$ its \emph{pro-nilpotent completion}, that is, its Hausdorff completion with respect to the filtration $(\gamma_n(H))_{n\in\NN}$. 
Observe that the nilpotent residual $\gamma_{\infty}(H):=\bigcap_{n\in\NN}\gamma_n(H)$ is the kernel of the canonical map $H\to \widehat{H}$.

\begin{corintro}\label{corintro:isom_completions}
The canonical maps $\widehat{\UU}_A(\QQ)\to\widehat{\U}^+_A(\QQ)\to\U^{ma+}_A(\QQ)$ are isomorphisms of topological groups.
\end{corintro}

By \cite{AM97}, the map $\UU_A(\QQ)\to \U^+_A(\QQ)$ is always surjective. 
The main result of \cite{DM07} implies that $\UU_A(\QQ)\to \U^+_A(\QQ)$ is also injective   provided $A$ is \emph{$3$-spherical}, that is, such that each submatrix $A_J=(a_{ij})_{i,j\in J}$ with $J\subseteq I$ of order $3$ is a Cartan matrix. The following theorem shows that the converse statement also holds:

\begin{thmintro}\label{thmintro:injiff3sph}
Let $\KK$ be an integral domain. If the map $\UU_A(\KK)\to \U^+_A(\KK)$ is injective, then $A$ is $3$-spherical.
\end{thmintro}

We refer to   \S\ref{section:prelim} for the definition of $\UU_A(\KK)$ and $\U^+_A(\KK)$ over general rings $\KK$.

As a second consequence of Theorem~\ref{thmintro:main}, we   obtain the following characterisation of the residual nilpotence  of $\UU_A(\QQ)$, which broadly generalises the first part of Theorem~\ref{thmintro:A-2-tilde} (note that $\U^+_A(\QQ)$ is always residually nilpotent, see Remark~\ref{rem:U+residually-nilp} below).

\begin{corintro}\label{corintro:residnilp}
The kernel of $\UU_A(\QQ)\twoheadrightarrow\U^+_A(\QQ)$ coincides with $\gamma_{\infty}(\UU_A(\QQ))$. 

In particular, the following conditions are equivalent:
\begin{enumerate}[(i)]
\item  $\UU_A(\QQ)$ is residually nilpotent.
\item  $\UU_A(\QQ)$ is residually torsionfree nilpotent. 
\item  $A$ is $3$-spherical.
\end{enumerate}

\end{corintro}

As in Theorem~\ref{thmintro:A-2-tilde}, we shall  construct explicit elements of $\gamma_{\infty}(\UU_A(\QQ))$, see Theorem~\ref{thm:freeproductKMS}.

Finally, as a third consequence of Theorem~\ref{thmintro:main}, we obtain normal forms for the elements of ($2$-spherical) minimal Kac--Moody groups over $\QQ$. Thanks to the Bruhat decomposition of $\G_A(\QQ)$, finding normal forms for the elements $g$ of $\G_A(\QQ)$ reduces to the case where $g\in\U^+_A(\QQ)$. On the other hand, elements $g$ of $\U^{ma+}_A(\QQ)$ do possess a normal form $g=\prod_{x\in\BBB}\exp(\lambda_xx)$ as an infinite product parametrised by an ordered $\QQ$-basis $\BBB$ of $\n^+_{A,\QQ}$ consisting of homogeneous elements $x$, with uniquely determined parameters $\lambda_x\in\QQ$ (see \cite[(8.33) p.212 or Theorem~8.51]{KMGbook}). Since $\U^+_A(\QQ)\subseteq\U^{ma+}_A(\QQ)$, this in principle also gives a normal form for the elements of $\U^+_A(\QQ)$. However, we would like to have a normal form that is ``intrinsic'' to $\U^+_A(\QQ)$ (in particular, the partial products of which the normal form is the limit should belong to $\U^+_A(\QQ)$), and there should be an algorithm allowing to transform the expression of an element of $\U^+_A(\QQ)$ as a product of generators in $\{\U_{\alpha_i}(\QQ) \ | \ i\in I\}$ into its normal form. Theorem~\ref{thmintro:main} allows to achieve this as follows.

Write $e_i$ ($i\in I$) for the Chevalley generators of $\g_{A,\QQ}$ generating $\n^+_{A,\QQ}$. Fix a $\QQ$-basis $\BBB$ of $\n^+_{A,\QQ}$ consisting of elements of the form $x:=[e_{i_1},[e_{i_2},[\dots,e_{i_n}]]]$ (in which case we set $n_x:=n\in\NN$), and for each $\lambda\in\QQ$ set $$u_x(\lambda):=[x_{\alpha_{i_1}}(\lambda),[x_{\alpha_{i_2}}(1),[\dots,x_{\alpha_{i_n}}(1)]]]\in\U^+_A(\QQ)$$ accordingly, where $x_{\alpha_i}\co\QQ\to \U_{\alpha_i}(\QQ)$ is the parametrisation of the root group $\U_{\alpha_i}(\QQ)$ associated with the simple root $\alpha_i$ ($i\in I$). Fix a total order on $\BBB$ such that $n_x<n_y\implies x<y$. 

\begin{corintro}\label{corintro:normalform}
Each $g\in\U^+_A(\QQ)$ can be uniquely written in the form
$$g=\prod_{x\in\BBB}u_x(\lambda_x)\in \widehat{\U}^+_A(\QQ)\quad\textrm{for some $\lambda_x\in\QQ$}$$
with $\prod_{n_x\geq n}u_x(\lambda_x)\in\gamma_n(\U^+_A(\QQ))$ for all $n\in\NN$.
\end{corintro}

Note that Corollary~\ref{corintro:normalform} becomes false if $A$ is not $2$-spherical (see Corollary~\ref{cor:normalform}).


\section{Preliminaries}\label{section:prelim}

\subsection{Group commutators}

Let $H$ be a group. Given elements $g_i\in H$, we write $[g_1,g_2]:=[g_1,g_2]\!]:=g_1\inv g_2\inv g_1g_2$. We define, inductively, the iterated commutator 
$$[g_1,g_2,\dots,g_n]\!]:=[g_1,[g_2,\dots,g_n]\!]], $$ 
and we also set $g^h:=h\inv gh$ for $g,h\in H$.

For each $n\in\NN$, we define inductively the subgroups $\gamma_n(H)$ of the lower central series of $H$ by $\gamma_1(H):=H$ and $\gamma_{n+1}(H)=[H,\gamma_n(H)]$. We also denote by $\gamma_{\infty}(H):=\bigcap_{n\geq 1}\gamma_n(H)$ the kernel of the canonical map $H\to \widehat{H}$, where $\widehat{H}$ is the {\bf pro-nilpotent completion} of $H$, that is, the Hausdorff completion of $H$ with respect to the filtration $(\gamma_n(H))_{n\in\NN}$. The group $H$ is {\bf residually nilpotent} (resp. {\bf residually torsionfree nilpotent}) if every nontrivial element $h\in H$ remains nontrivial in some nilpotent (resp. torsionfree nilpotent) quotient of $H$. Thus, $H$ is residually nilpotent if and only if $\gamma_{\infty}(H)=\{1\}$.

Let $(H_n)_{n\geq 1}$ be a sequence of normal subgroups of $H$ with $H_1=H$, $H_n\supseteq H_{n+1}$ and $[H_m,H_n]\subseteq H_{m+n}$ for all $m,n\geq 1$ (for instance, one could take $H_n=\gamma_n(H)$ for all $n$). Note then that $\gamma_n(H)\subseteq H_n$ for all $n\geq 1$. 

The space $\Lie(H;(H_n)_{n\geq 1}):=\bigoplus_{n\geq 1}H_n/H_{n+1}$  is a Lie ring, with respect to the Lie bracket $$[xH_{n+1},yH_{m+1}]:=[x,y]H_{m+n+1}\quad\textrm{for all $x\in H_{n}$ and $y\in H_{m}$.}$$ This follows from the usual commutator identities
\begin{align}
&[x,zy]=[x,y]\cdot [x,z]^y\quad\textrm{and}\quad [xz,y]=[x,y]^z\cdot [z,y]\quad\textrm{for all $x,y,z\in H$,}\label{eqn:commutatorproduct}\\
&[[x,y\inv],z]^y=[y,[z,x\inv]]^x\cdot [x,[y,z\inv]]^z\quad\textrm{for all $x,y,z\in H$,}\label{eqn:HallWitt}
\end{align}
which imply that for all $x\in H_m$, $y\in H_n$, $z\in H_p$ and $r\in\ZZ$,
\begin{align}
&[x^r,y]\equiv [x,y]^r\equiv [x,y^r] \ \modulo H_{m+n+1},\label{eqn:commutatorpower}\\
&[x,[y,z]]\cdot [y,[z,x]]\cdot [z,[x,y]]\equiv 1 \modulo H_{m+n+p+1}.
\end{align}
We set for short $\Lie(H):=\Lie(H; (\gamma_n(H))_{n\geq 1})$. 

Note that any group morphism $G\to H$ induces a Lie ring morphism $\Lie(G)\to\Lie(H;(H_n)_{n\geq 1})$ since $\gamma_n(G)$ is mapped into $\gamma_n(H)\subseteq H_n$ for all $n$.


\subsection{Kac--Moody algebras}\label{subsection:KMA}

Let $A=(a_{ij})_{i,j\in I}$ be a {\bf generalised Cartan matrix} (hereafter abbreviated \textbf{GCM}), and let $\Delta=\Delta(A)$ be the corresponding Kac--Moody root system, with set of {\bf simple roots} $\Pi=\{\alpha_i \ | \ i\in I\}$ (see \cite[\S 1.1--1.3]{Kac}). Set $Q:=\sum_{i\in I}\ZZ\alpha_i$ and $Q_+:=\sum_{i\in I}\NN\alpha_i$. For $\alpha=\sum_{i\in I}n_i\alpha_i\in Q$, we let $\height(\alpha):=\sum_{i\in I}n_i\in\ZZ$ denote its {\bf height}.

Let $\n^+_{\QQ}=\n^+_{A,\QQ}$ denote the $\QQ$-Lie algebra defined by the presentation with generators $e_i$ ($i\in I$) and relations
\begin{equation}
(\ad e_i)^{1-a_{ij}}e_j=0\quad\textrm{for all $i,j\in I$ with $i\neq j$}.
\end{equation}
Then $\n^+_{\QQ}$ admits a $Q_+$-grading
$$\n^+_{\QQ}=\bigoplus_{\alpha\in Q_+}\g_{\alpha}$$
defined by letting $e_i$ have degree $\alpha_i$ for each $i\in I$. Moreover, $\g_{\alpha}\neq \{0\}$ if and only if $\alpha\in\Delta^+:=\Delta\cap Q_+$. For a nonzero element $x\in\g_{\alpha}$, we set $\deg(x):=\alpha$.

The {\bf Weyl group} of $A$ is the group $W=W_A$ of $\ZZ$-linear automorphisms of $Q$ generated by the {\bf simple reflections}
$$s_i\co Q\to Q:\alpha_j\mapsto \alpha_j-a_{ij}\alpha_i\quad\textrm{for $i,j\in I$}.$$
It is a Coxeter group with Coxeter generating set $S=\{s_i \ | \ i\in I\}$, where the order of $s_is_j$ for $i\neq j$ is given by $m_{ij}=2,3,4,6$ or $\infty$, according to whether $a_{ij}a_{ji}=0,1,2,3$ or $\geq 4$ respectively (see \cite[Chapter~3]{Kac}). The set $\Delta^{re}:=W\Pi\subseteq Q$ of {\bf real roots} is contained in $\Delta$, and we set $\Delta^{re+}:=\Delta^{re}\cap \Delta^+$.

The GCM $A$ is {\bf spherical} if $W_A$ is finite (that is, if it is a Cartan matrix). For $J\subseteq I$, set $A_J:=(a_{ij})_{i,j\in J}$. For $r\geq 2$, we call $A$ {\bf $r$-spherical} if $A_J$ is spherical for all $J\subseteq I$ with $|J|\leq r$. The cardinality of $I$ is called the {\bf rank} of $A$. If $A$ is spherical of rank $2$, then $\Delta$ is a rank $2$ root system of type $A_1\times A_1$, $A_2$, $B_2$, or $G_2$, and we also say that $A$ is of that {\bf type}.


\subsection{Minimal Kac--Moody groups}\label{subsection:MKMG}
Let $A=(a_{ij})_{i,j\in I}$ be a GCM and let $\KK$ be a (commutative, unital) ring. Let $\G_A(\KK)$ be the constructive Tits functor of (simply connected) type $A$ over $\KK$, as in \cite{Tits87} or \cite[Definition~7.47]{KMGbook}. It is a certain amalgamated product of the  {\bf root groups} $\U_{\alpha}(\KK)\cong (\KK,+)$ indexed by the real roots $\alpha\in\Delta^{re}$. We denote by $x_{\alpha}\co \KK\to\U_{\alpha}(\KK)$ the group isomorphism parametrising $\U_{\alpha}(\KK)$. We also set $$\U^+_A(\KK):=\langle \U_{\alpha}(\KK) \ | \ \alpha\in\Delta^{re+}\rangle\subseteq\G_A(\KK).$$

For $i\in I$, we set $\widetilde{s}_i:=x_{\alpha_i}(1)x_{-\alpha_i}(1)x_{\alpha_i}(1)\in\G_A(\KK)$. Then 
\begin{equation}\label{eqn:siactiononrootgroups}
\widetilde{s}_i\U_{\alpha}(\KK)\widetilde{s}_i\inv=\U_{s_i\alpha}(\KK)\quad\textrm{for all $\alpha\in\Delta^{re}$ and $i\in I$.}
\end{equation}
For $i\in I$ and $a\in\KK$, we will also write for short $x_i(a):=x_{\alpha_i}(a)$ and $x_i:=x_i(1)$.


\subsection{Maximal Kac--Moody groups}
Let $A=(a_{ij})_{i,j\in I}$ be a GCM. Let $\U^{ma+}_A$ be the affine group scheme of type $A$ defined in \cite[\S 8.5.1--8.5.2]{KMGbook}. The group $\U^{ma+}_A(\QQ)$ can be described as follows. For each $\alpha\in\Delta^+$, choose a $\QQ$-basis $\BBB_{\alpha}$ of $\g_{\alpha}$, and set $\BBB=\bigcup_{\alpha\in\Delta_+}\BBB_{\alpha}$. Choose a total order on $\BBB$ (say, such that $\height(\deg(x))<\height(\deg(y))\implies x<y$).
Consider the completion $\widehat{\UU}_{\QQ}(\n^+_{\QQ})$ of the universal enveloping algebra $\UU_{\QQ}(\n^+_{\QQ})$ of $\n^+_{\QQ}$ with respect to its natural $\NN$-gradation, induced by the $\NN$-gradation $\n^+_{\QQ}=\bigoplus_{n\in\NN}(\oplus_{\height(\alpha)=n}\g_{\alpha})$ of $\n^+_{\QQ}$. Then $\U^{ma+}_A(\QQ)$ is the subgroup of the multiplicative group of invertible elements $g$ of $\widehat{\UU}_{\QQ}(\n^+_{\QQ})$ that can be written as a (possibly infinite) product
\begin{equation}
g=\prod_{x\in\BBB}\exp(\lambda_xx)\in\widehat{\UU}_{\QQ}(\n^+_{\QQ})\quad\textrm{for some $\lambda_x\in \QQ$.}
\end{equation}
Moreover, each $g\in \U^{ma+}_A(\QQ)$ has a unique expression in the form of such a product. The group $\U^{ma+}_A(\QQ)$ is a complete Hausdorff topological group for the topology defined by the filtration $(\U^{ma}_n(\QQ))_{n\in\NN}$, where $\U^{ma}_n(\QQ)$ is the normal subgroup of $\U^{ma+}_A(\QQ)$ defined by
$$\U^{ma}_n(\QQ):=\Big\{\prod_{x\in\BBB, \ \height(\deg(x))\geq n}\exp(\lambda_xx) \ \Big| \ \lambda_x\in\QQ\Big\}\subseteq \U^{ma+}_A(\QQ).$$

There is an injective group morphism 
\begin{equation}\label{eq:U+->Uma+}
\U^+_A(\QQ)\to\U^{ma+}_A(\QQ)
\end{equation}
mapping $\U_{\alpha}(\QQ)$ to $\exp(\g_{\alpha})$ for each $\alpha\in\Delta^{re+}$, and one chooses the basis $\BBB_{\alpha}=\{e_{\alpha}\}$ for $\alpha\in\Delta^{re+}$ so that $x_{\alpha}(r)$ is mapped to $\exp(re_{\alpha})$ for all $r\in\QQ$. 
We then identify $\U^+_A(\QQ)$ with a subgroup of $\U^{ma+}_A(\QQ)$; by \cite[Corollary~8.91]{KMGbook}, it is dense in $\U^{ma+}_A(\QQ)$.

Since $[\U^{ma}_n(\QQ),\U^{ma}_m(\QQ)]\subseteq \U^{ma}_{m+n}(\QQ)$ for all $m,n$ by \cite[Lemma~8.58(5)]{KMGbook}, the space $$\Lie(\U^{ma+}_A(\QQ);(\U^{ma}_n(\QQ))_{n\geq 1})=\bigoplus_{n\geq 1}\U^{ma}_n(\QQ)/\U^{ma}_{n+1}(\QQ)$$
is a Lie ring, and the map (\ref{eq:U+->Uma+}) induces a  
Lie ring morphism
\begin{equation*}
\Lie(\U^+_A(\QQ))\to \bigoplus_{n\geq 1}\U^{ma}_n(\QQ)/\U^{ma}_{n+1}(\QQ).
\end{equation*}

Observe also that $\bigoplus_{n\geq 1}\U^{ma}_n(\QQ)/\U^{ma}_{n+1}(\QQ)$ is naturally a $\QQ$-Lie algebra isomorphic to $\n^+_{\QQ}$, where the isomorphism is given by the assignment 
\begin{equation*}
\exp(\lambda_xx)\U^{ma}_{n+1}(\QQ)\mapsto \lambda_xx\in \n^+_{\QQ}
\end{equation*}
for $x\in\BBB$ with $\height(\deg(x))=n$ and $\lambda_x\in\QQ$ (see  \cite[Lemma~8.58(6)]{KMGbook}). In particular, the map (\ref{eq:U+->Uma+}) induces a Lie ring morphism
\begin{equation}\label{eqn:morphismton}
\Lie(\U^+_A(\QQ))\to\n^+_{\QQ}:x_i(r)\gamma_2(\U^+_A(\QQ))\mapsto re_i\quad\textrm{for all $i\in I$ and $r\in\QQ$.}
\end{equation}

\begin{remark}\label{rem:U+residually-nilp}
Since $\gamma_n(\U^+_A(\QQ))\subseteq \U^{ma}_n(\QQ)$ for all $n$ and $\bigcap_{n\geq 1}\U^{ma}_n(\QQ)=\{1\}$, the group $\U^+_A(\QQ)$ is residually nilpotent
\end{remark}


\subsection{Kac--Moody--Steinberg groups}\label{subsection:KMSG}
Let $A=(a_{ij})_{i,j\in I}$ be a GCM, and let $\KK$ be a ring. Assume that $A$ is $2$-spherical. Thus, for each $i,j\in I$ with $i\neq j$, the rank $2$ GCM $A_{ij}:=A_{\{i,j\}}$ is spherical, of one of the types $A_1\times A_1$, $A_2$, $B_2$, or $G_2$. For $i\neq j$, we consider the subgroup 
\begin{equation}
U_{ij}:=\U^+_{A_{ij}}(\KK)=\langle \U_{\gamma}(\KK) \ | \ \gamma\in \Delta^{+}_{ij}:=\Delta^{+}\cap(\NN\alpha_i+\NN\alpha_j)\rangle
\end{equation}
 of $\U^+_A(\KK)$, as well as its subgroups $U_i:=\U_{\alpha_i}(\KK)$ and $U_j:=\U_{\alpha_j}(\KK)$. We define the {\bf KMS group} $\UU_A(\KK)$ of type $A$ over $\KK$ as the inductive limit of the inductive system of groups $\{U_i, U_{ij} \ | \ i,j\in I, \ i\neq j\}$ with respect to the inclusion maps $U_i\to U_{ij}$ and $U_j\to U_{ij}$. 

There is a natural group morphism
\begin{equation}\label{eqn:KMStoU+}
\UU_A(\KK)\to \U_A^+(\KK)
\end{equation}
mapping $U_i$ to $\U_{\alpha_i}(\KK)$ and $U_{ij}$ to $\U^+_{A_{ij}}(\KK)$. In particular, the subgroups $U_i$ and $U_{ij}$ inject in $\UU_A(\KK)$ (as they do in $\U_A^+(\KK)$), and for $\gamma\in \Delta^{+}_{ij}$, we again use the notation $x_{\gamma}\co \KK\to \UU_A(\KK)$ for the parametrisation of $\U_{\gamma}(\KK)\subseteq U_{ij}\subseteq \UU_A(\KK)$, as well as its shorthands $x_i(a):=x_{\alpha_i}(a)$ and $x_i:=x_i(1)$ for $i\in I$ and $a\in\KK$.
Thus, the map (\ref{eqn:KMStoU+}) is given by the assignment $x_{\gamma}(a)\mapsto x_{\gamma}(a)$ for all $a\in\KK$ and $\gamma\in \Delta^{+}_{ij}$ ($i\neq j$).

When $\KK$ is a field of order $\geq 4$, both $\UU_A(\KK)$ and $\U^+_A(\KK)$ are generated by the simple root groups $x_i(\KK)$ ($i\in I$), and hence the map (\ref{eqn:KMStoU+}) is surjective (see \cite{AM97}). When $\KK$ is a field of order $\geq 5$ and $A$ is $3$-spherical, then the map (\ref{eqn:KMStoU+}) is also injective, as follows from \cite{DM07} (see also \cite[Prop.~3.7]{KMSHDX} for a more explicit proof). In particular, if $\KK=\QQ$, then the map (\ref{eqn:KMStoU+}) is surjective, and an isomorphism when $A$ is $3$-spherical.


\subsection{Coset graphs}\label{subsection:CG}

Let $A=(a_{ij})_{i,j\in I}$ be a $2$-spherical GCM and let $\KK$ be a ring. Let $i,j\in I$ with $i\neq j$. The {\bf coset graph} $\Gamma(U_{ij}; U_i, U_j)$ is the bipartite graph with vertex set $U_{ij}/U_i \sqcup U_{ij}/U_j$ and edge set $U_{ij}$, where the edge $g\in U_{ij}$ joins the vertices $gU_i$ and $gU_j$. In other words, two distinct edges $g,h\in U_{ij}$ are incident if and only if $g=hx_k(a)$ for some $k\in\{i,j\}$ and $a\in \KK$.

In particular, the distance in $\Gamma(U_{ij}; U_i, U_j)$ from an edge $g\in U_{ij}$ to the base edge represented by the neutral element $1$ is the smallest number $n$ such that $g$ can be written as a word of length $n$ in the elements of $U_i \cup U_j$, i.e. such that $g$ has the form
$$g = x_{i'}(a_1) x_{j'}(a_2) x_{i'}(a_3) x_{j'}(a_4) \dots , $$
where $\{i',j'\}=\{i,j\}$ and $a_r \in \KK$ for all $r \in \{1, \dots, n\}$. 

Similarly, the girth (i.e. the length of a shortest cycle) of $\Gamma(U_{ij}; U_i, U_j)$ is the smallest number $n\geq 1$ such that 
$$1 = x_{i'}(a_1) x_{j'}(a_2) x_{i'}(a_3) x_{j'}(a_4) \dots , $$
where $\{i',j'\}=\{i,j\}$ and $a_r \in \KK\setminus\{0\}$ for all $r \in \{1, \dots, n\}$.


\subsection{Spherical rank $2$ subgroups}\label{subsection:SR2S}

We recall that if $A=\begin{psmallmatrix}2&a_{12}\\ a_{21}&2\end{psmallmatrix}$ is a spherical GCM of rank $2$ and $\KK$ is a ring, and if $\Delta^+=\{\gamma_1,\dots,\gamma_n\}$ is an enumeration of $\Delta^+=\Delta^{re+}$, then any element $g\in\U^+_A(\KK)=\UU_A(\KK)$ has a unique expression as a product 
\begin{equation}\label{eqn:normalformrank2}
g=\prod_{i=1}^nx_{\gamma_i}(a_i)\quad\textrm{with $a_1,\dots,a_n\in\KK$}.
\end{equation}
Moreover, $\UU_A(\KK)$ has a presentation with generators the subgroups $\U_{\gamma}(\KK)$ ($\gamma\in\Delta^+$) and relations given by commutator relations of the form
\begin{equation}\label{eqn:commutationrel}
[x_{\alpha}(a),x_{\beta}(b)]=\prod_{\gamma=r\alpha+s\beta\in ]\alpha,\beta[}x_{\gamma}(C^{\alpha\beta}_{rs}a^rb^s)\quad\textrm{for all $a,b\in\KK$}
\end{equation}
for each pair $\{\alpha,\beta\}\subseteq\Delta^+$ of distinct roots, where $$]\alpha,\beta[:=\{r\alpha+s\beta\in\Delta^+ \ | \ r,s\geq 1\}$$ and the $C^{\alpha\beta}_{rs}$ are integers determined by $A$ (see \cite[Exp.~XXIII, \S3]{SGA}).

Finally, note that the commutation relations (\ref{eqn:commutationrel}) imply that 
\begin{equation}\label{eqn:SerreinLieUA}
[x_i(\KK),\dots,x_i(\KK),x_j(\KK)]\!]\in\gamma_{|a_{ij}|+3}(\UU_A(\KK))\quad\textrm{($|a_{ij}|+1$ terms ``$x_i(\KK)$'')}
\end{equation} 
for $(i,j)\in\{(1,2),(2,1)\}$.


\section{Free amalgamated products}\label{section:FAP}
For subgroups $C,H$ of a group $G$, define recursively the subgroups $C^{(n)}(H)$ of $G$ for $n\in \NN$ by $C^{(0)}(H):=H$ and $C^{(n+1)}(H):=[C,C^{(n)}(H)]$. 

\begin{lemma}\label{lemma:CHG}
Let $C,H$ be subgroups of a group $G$. Assume that $[C,H]\subseteq H$ and that $C^{(N)}(H)\subseteq\gamma_2(H)$ for some $N\geq 1$. Then 
$C^{(Nn-n+1)}(\gamma_n(H))\subseteq\gamma_{n+1}(H)$ 
for all $n\geq 1$.
\end{lemma}
\begin{proof}
By hypothesis, the group $C$ normalises $H$. Therefore we have  $[C,\gamma_n(H)]\subseteq \gamma_n(H)$ for all $n\in\NN$. We prove the claim by induction on $n\geq 1$. For $n=1$, this holds by assumption. Assume now that 
\begin{equation}\label{eqn:CNnn1induction}
C^{(Nn-n+1)}(\gamma_n(H))\subseteq\gamma_{n+1}(H)
\end{equation}
for some $n\geq 1$.

It follows from the Hall--Witt identity (\ref{eqn:HallWitt}) that for all $i,j\in\NN$,
\begin{align*}
[C,[C^{(i)}(H),C^{(j)}(\gamma_n(H))]]\subseteq & [C^{(i+1)}(H),C^{(j)}(\gamma_n(H))]\cdot[C^{(i)}(H),C^{(j+1)}(\gamma_n(H))]\cdot\\
&[C,[C^{(i)}(H),C^{(j+1)}(\gamma_n(H))]]\cdot\gamma_{n+2}(H).
\end{align*}
Applying the above formula repeatedly to its right-hand side and using the fact that $C^{(r+1)}(H)\subseteq C^{(r)}(H)$ and $C^{(r+1)}(\gamma_n(H))\subseteq C^{(r)}(\gamma_n(H))$ for all $r\in\NN$, we deduce that for all $i,j\in\NN$,
\begin{align*}
[C,[C^{(i)}(H),C^{(j)}(\gamma_n(H))]]\subseteq & [C^{(i+1)}(H),C^{(j)}(\gamma_n(H))]\cdot[C^{(i)}(H),C^{(j+1)}(\gamma_n(H))]\cdot\\
&[C,[C^{(i)}(H),C^{(j+r)}(\gamma_n(H))]]\cdot\gamma_{n+2}(H)
\end{align*}
for all $r\geq 1$. In view of (\ref{eqn:CNnn1induction}), it follows that 
\begin{align*}
[C,[C^{(i)}(H),C^{(j)}(\gamma_n(H))]]\subseteq & [C^{(i+1)}(H),C^{(j)}(\gamma_n(H))]\cdot[C^{(i)}(H),C^{(j+1)}(\gamma_n(H))]\cdot\\
&\gamma_{n+2}(H)
\end{align*}
for all $i,j\in\NN$. Using (\ref{eqn:commutatorproduct}), we then get by induction on $r\geq 1$ that
\begin{equation}
C^{(r)}(\gamma_{n+1}(H))\subseteq \prod_{i+j=r}[C^{(i)}(H),C^{(j)}(\gamma_n(H))]\cdot\gamma_{n+2}(H).
\end{equation}
Since $C^{(N)}(H)\subseteq\gamma_2(H)$ and in view of (\ref{eqn:CNnn1induction}),
we conclude that 
$$C^{((Nn-n+1)+N-1)}(\gamma_{n+1}(H))\subseteq\gamma_{n+2}(H).$$
This completes the induction step.
\end{proof}

\begin{lemma}\label{lemma:rn}
Let $G$ be a group, and let $H,C$ be subgroups of $G$ such that $G=H\rtimes C$. Assume that $C$ is abelian and that $C^{(N)}(H)\subseteq\gamma_2(H)$ for some $N\geq 2$. Then $$\gamma_{r_n}(G)\subseteq\gamma_{n+1}(H)\quad\textrm{for all $n\geq 1$},$$
where $r_n:=1+n+(N-1)n(n+1)/2$.
\end{lemma}
\begin{proof}
It follows from \cite[Theorem~1.1]{GP20} that $\gamma_n(H\rtimes C)=L_n\rtimes \gamma_n(C)$ for all $n\geq 1$, where the subgroup $L_n$ is defined inductively by $L_1:= H$ and, for $n\geq 2$,
\begin{equation}\label{eqn:Lndef}
L_{n}=\langle [\gamma_{n-1}(C),H], [C,L_{n-1}], [H,L_{n-1}]\rangle\subseteq H.
\end{equation}
Note that $r_n\geq r_1=N+1>2$ for all $n\geq 1$. In particular, $\gamma_{r_n}(G)=L_{r_n}$, and it suffices to show that 
\begin{equation}\label{eqn:Lrn}
L_{r_n}\subseteq \gamma_{n+1}(H)\quad\textrm{for all $n\geq 1$.}
\end{equation}

We prove (\ref{eqn:Lrn}) by induction on $n\geq 1$. For $n=1$, we have to show that $L_{N+1}\subseteq\gamma_2(H)$. By (\ref{eqn:Lndef}), we have $L_2=[C,H]\cdot\gamma_2(H)$ and, since $\gamma_2(C)=\{1\}$ by assumption, 
$$L_{r+1}\subseteq [C,L_r]\cdot\gamma_2(H)\quad\textrm{for all $r\geq 2$}.$$
Hence $L_{N+1}\subseteq C^{(N)}(H)\cdot\gamma_2(H)$, yielding the claim. Assume now that (\ref{eqn:Lrn}) holds for some $n\geq 1$. Since $r_n>2$ and $\gamma_2(C)=1$, it follows from (\ref{eqn:Lndef}) that $L_{m}=\langle [C,L_{m-1}], [H,L_{m-1}]\rangle$ for all $m\geq r_n$, and hence by induction hypothesis
$$L_m\subseteq [C,L_{m-1}]\cdot \gamma_{n+2}(H)\quad\textrm{for all $m\geq r_n+1$.}$$
In particular, $$L_{r_{n+1}}\subseteq C^{(r_{n+1}-r_n)}(L_{r_n})\cdot\gamma_{n+2}(H)\subseteq C^{((N-1)(n+1)+1)}(\gamma_{n+1}(H))\cdot\gamma_{n+2}(H),$$
where the second inclusion follows from the induction hypothesis. Since $$C^{((N-1)(n+1)+1)}(\gamma_{n+1}(H))\subseteq \gamma_{n+2}(H)$$ by Lemma~\ref{lemma:CHG}, this completes the induction step.
\end{proof}

\begin{prop}
Let $A,B$ be subgroups of a group $G$, and suppose that $A=A'\rtimes C$ and $B=B'\rtimes C$ for some subgroups $A',B',C$. Assume that $A',B'$ are residually torsionfree nilpotent and that $C$ is abelian. Assume, moreover, that $C^{(N)}(A)=C^{(N)}(B)=\{1\}$ for some $N\geq 1$. Then the amalgamated product $A*_CB$ is residually nilpotent. 
\end{prop}
\begin{proof}
Note that $A*_CB\cong (A'*B')\rtimes C$ (since these two groups admit the same presentation). Note also that $C^{(N)}(A'*B')\subseteq \gamma_2(A'*B')$. The proposition then follows from Lemma~\ref{lemma:rn} and the fact that $A'*B'$ is residually nilpotent by a result of Mal'cev (see \cite{Mal49} or \cite[Theorem~1.2]{Baum99}).
\end{proof}

 Theorem~\ref{thmintro:free-amalgamated-prod} now follows as an immediate consequence.


\section{Groups generated by copies of $\QQ$}
Throughout this section, we fix a set $I$ and we let $\KK$ denote a (unital) subring of $\QQ$. We also let $H$ be a group generated by copies of $(\KK,+)$ parametrised by $x_i\co\KK\to H$ for $i\in I$, and we set $\Gamma_m:=\gamma_m(H)$ for all $m\geq 1$. We again use the shorthand $x_i:=x_i(1)$ for $i\in I$.

\begin{lemma}\label{lemma:scalarmultQ}
Let $i_1,\dots,i_n\in I$ and $a_1,\dots,a_n\in\KK$. Then 
\begin{align*}
[x_{i_1}(a_1),x_{i_2}(a_2),\dots, x_{i_n}(a_n)]\!]\Gamma_{n+1}= [x_{i_1}(a_1\dots a_n),x_{i_2},\dots,x_{i_n}]\!] \Gamma_{n+1}.
\end{align*}
\end{lemma}
\begin{proof}
Reasoning inductively, it is sufficient to prove that 
\begin{align*}
[x_{i_1}(a_1),x_{i_2}(a_2\dots a_n),x_{i_3},\dots,x_{i_n}]\!]\Gamma_{n+1}= [x_{i_1}(a_1\dots a_n),x_{i_2},\dots,x_{i_n}]\!]\Gamma_{n+1}.
\end{align*}
Write $a_2\dots a_n=r/s$ with $r,s\in\ZZ\setminus\{0\}$. Then using (\ref{eqn:commutatorpower}) repeatedly, we get
\begin{align*}
[x_{i_1}(a_1),x_{i_2}(a_2\dots a_n),x_{i_3},\dots,x_{i_n}]\!]&\equiv [x_{i_1}(a_1/s)^s,[x_{i_2}(1/s),x_{i_3},\dots,x_{i_n}]\!]^r]\\
&\equiv [x_{i_1}(a_1/s)^r,[x_{i_2}(1/s),x_{i_3},\dots,x_{i_n}]\!]^s]\\
&\equiv [x_{i_1}(a_1\dots a_n),x_{i_2},\dots,x_{i_n}]\!] \ \modulo \Gamma_{n+1}.\qedhere
\end{align*}
\end{proof}

\begin{prop}\label{prop:GGCQscalarmult}
Assume that for each $s\in\NN$ invertible in $\KK$, the assignment $x_i(a)\mapsto x_i(a/s)$ for $i\in I$ and $a\in\KK$ extends to a group morphism $H\to H$. Then the Lie ring 
$$\Lie(H)=\bigoplus_{n\geq 1}\Gamma_n/\Gamma_{n+1}$$
is a $\KK$-Lie algebra with respect to the scalar multiplication defined on generators by
\begin{equation}\label{eqn:defscalarmult}
\lambda\cdot [x_{i_1}(a_1),\dots,x_{i_n}(a_n)]\!]\Gamma_{n+1}:=[x_{i_1}(\lambda a_1),x_{i_2}(a_2),\dots,x_{i_n}(a_n)]\!]\Gamma_{n+1}
\end{equation}
 for all $\lambda,a_{1},\dots,a_{n}\in\KK$ and $i_1,\dots,i_n\in I$.
\end{prop}
\begin{proof}
To check that the given scalar multiplication indeed yields a well-defined $\KK$-module structure on $\Lie(H)$, we have to show that if $\lambda\in\KK$ and $\prod_iu_i$ is a product of elements $u_i=[x_{i_1}(a^{(i)}_1),\dots,x_{i_n}(a^{(i)}_n)]\!]$ that belongs to $\Gamma_{n+1}$, then the corresponding product of the elements $[x_{i_1}(\lambda a^{(i)}_1),\dots,x_{i_n}(a^{(i)}_n)]\!]$ also belongs to $\Gamma_{n+1}$.
Since
$$[x_{i_1}(ra_1),\dots,x_{i_n}(a_n)]\!]\equiv [x_{i_1}(a_1),\dots,x_{i_n}(a_n)]\!]^{r} \ \modulo \Gamma_{n+1}$$
for all $r\in\ZZ$ by (\ref{eqn:commutatorpower}), it is sufficient to prove this claim for $\lambda$ of the form $\lambda=1/s$ with $s\in\NN^*$ invertible in $\KK$. 

By assumption, there is a group morphism $\pi_s\co H\to H$ such that $\pi_s(x_i(a))=x_i(a/s)$ for all $i\in I$ and $a\in\KK$. Then $\pi_{s}(\prod_iu_i)\in \Gamma_{n+1}$, and hence by Lemma~\ref{lemma:scalarmultQ} and (\ref{eqn:commutatorpower}),
\begin{align*}
\prod_i[x_{i_1}(a^{(i)}_1/s),\dots,x_{i_n}(a^{(i)}_n)]\!]&\equiv 
\prod_i[x_{i_1}(s^{n-1}a^{(i)}_1/s),x_{i_2}(a^{(i)}_2/s),\dots,x_{i_n}(a^{(i)}_n/s)]\!]\\
&\equiv \prod_i[x_{i_1}(a^{(i)}_1/s),\dots,x_{i_n}(a^{(i)}_n/s)]\!]^{s^{n-1}}\\
&\equiv \pi_{s}\big(\prod_iu_i\big)^{s^{n-1}}\equiv 1 \ \modulo \Gamma_{n+1},
\end{align*}
as desired.

The $\KK$-bilinearity of the Lie bracket on $\Lie(H)$ follows from Lemma~\ref{lemma:scalarmultQ}.
\end{proof}

\begin{corollary}\label{corollary:conditionequivrtfnrn}
Let $I$ be a set and let $H$ be a group generated by copies of $(\QQ,+)$ parametrised by $x_i\co\QQ\to H$ for $i\in I$. Assume that for each $s\in\NN^*$, the assignment $x_i(a)\mapsto x_i(a/s)$ for $i\in I$ and $a\in\QQ$ extends to a group morphism $H\to H$. Then $H/\gamma_n(H)$ is torsionfree for all $n\geq 1$. 

In particular, $H$ is residually torsionfree nilpotent if and only if it is residually nilpotent.
\end{corollary}
\begin{proof}
By Proposition~\ref{prop:GGCQscalarmult}, $\Lie(H)$ admits a $\QQ$-Lie algebra structure. In particular, $\gamma_n(H)/\gamma_{n+1}(H)$ is torsionfree for all $n\geq 1$, yielding the claim.
\end{proof}


\section{Pronilpotent completions}
Let $A=(a_{ij})_{i,j\in I}$ be a GCM. Unless otherwise stated, we assume $A$ to be $2$-spherical. In order to prove Theorem~\ref{thmintro:main}, we first show that $\Lie(\UU_A(\QQ))$ and $\Lie(\U^+_A(\QQ))$ admit a natural $\QQ$-Lie algebra structure.

\begin{lemma}\label{lemma:ringmorphismsigma}
Let $\KK$ be a ring, and let $\lambda\in\KK$. Then there is a group morphism $\pi_{\lambda}\co\UU_A(\KK)\to\UU_A(\KK)$ mapping $x_{\gamma}(a)$ to $x_{\gamma}(\lambda^{\height(\gamma)}a)$ for all $\gamma\in\Delta^+_{ij}$, $i,j\in I$ with $i\neq j$, and $a\in\KK$. 
\end{lemma}
\begin{proof}
If $A$ is of rank $2$, this readily follows from the presentation of $\UU_A(\KK)$ (see \S\ref{subsection:SR2S}). For a general $A$, since $\pi_{\lambda}$ preserves each subgroup $U_i$ and $U_{ij}$ and commutes with the inclusion maps $U_i\to U_{ij}$ and $U_j\to U_{ij}$, this follows from the definition of $\UU_A(\KK)$.
\end{proof}

\begin{prop}\label{prop:UKMSscalarmult}
The formulas (\ref{eqn:defscalarmult}) define a $\QQ$-Lie algebra structure on $\Lie(\UU_A(\QQ))$.
\end{prop}
\begin{proof}
Recall from \S\ref{subsection:KMSG} that $\UU_A(\QQ)$ is generated by its subgroups $x_i(\QQ)$ for $i\in I$.
The claim thus follows from Proposition~\ref{prop:GGCQscalarmult} and Lemma~\ref{lemma:ringmorphismsigma}.
\end{proof}

\begin{remark}\label{rmk:res-nilp-IFF-tf-res-nilp}
In view of Corollary~\ref{corollary:conditionequivrtfnrn}, Lemma~\ref{lemma:ringmorphismsigma} also implies that $\UU_A(\QQ)$ is residually torsionfree nilpotent if and only if it is residually nilpotent. For similar reasons, free amalgamated products of the form $U_{12}*_{U_2}U_{23}$ over $\QQ$ are residually nilpotent by Theorem~\ref{thmintro:free-amalgamated-prod}, and hence residually torsionfree nilpotent.
\end{remark}

From now on, we equip $\Lie(\UU_A(\QQ))$ with its $\QQ$-Lie algebra structure provided by Proposition~\ref{prop:UKMSscalarmult}.

\begin{prop}\label{prop:KLiealgtemp}
The assignment $e_i\mapsto x_i\gamma_2(\UU_A(\QQ))$ extends to a $\QQ$-Lie algebra isomorphism
\begin{equation*}
\n^+_{\QQ}\to\Lie(\UU_A(\QQ)).
\end{equation*}
\end{prop}
\begin{proof}
Recall that $\n^+_{\QQ}$ has a presentation with generators $e_i$ ($i\in I$) and relations $(\ad e_i)^{|a_{ij}|+1}e_j$ ($i\neq j$). Since $[x_i,x_i,\dots,x_j]\!]\in\gamma_{|a_{ij}|+3}(\UU_A(\QQ))$ ($|a_{ij}|+1$ terms ``$x_i$'') by definition of $\UU_A(\QQ)$ and (\ref{eqn:SerreinLieUA}), the assignment $e_i\mapsto x_i\gamma_2(\UU_A(\QQ))$ extends to a surjective $\QQ$-Lie algebra morphism $\pi\co\n^+_{\QQ}\to\Lie(\UU_A(\QQ))$. 
On the other hand, the natural group morphism $\UU_A(\QQ)\to  \U_A^+(\QQ)\to \U^{ma+}_A(\QQ)$ induces a Lie ring morphism $\Lie(\UU_A(\QQ))\to \Lie(\U^+_A(\QQ))\to\n^+_{\QQ}$ (see (\ref{eqn:morphismton})) whose precomposition with $\pi$ is the identity on $\n^+_{\QQ}$. In particular, $\pi$ is injective, and hence a $\QQ$-Lie algebra isomorphism.
\end{proof}

We shall now describe how  the above arguments also yield simpler presentations of the rank $2$ subgroups $U_{ij}=\UU_{A_{ij}}(\QQ)$ of $\UU_A(\QQ)$ with generators the copies $U_i$ and $U_j$ of $(\QQ,+)$ (and therefore also a simpler presentation of $\UU_A(\QQ)$, on the generators $U_i$ for $i\in I$). Recall from \S\ref{section:FAP} the iterated commutator subgroup notation $C^{(n)}(H)=[C,[C,\dots,H]]$ for subgroups $C,H$ of a group $G$.

\begin{lemma}\label{lem:presentation-rank-2}
Assume that $A$ is spherical of rank $2$, indexed by $I=\{i,j\}$. Let $U_i$ and $U_j$ be copies of $(\QQ,+)$, and let $\mathcal R$ be a collection of subgroups of $U_i*U_j$ of the form $[U_{i_1},U_{i_2},\dots,U_{i_r}]\!]$ with $i_1,\dots, i_r\in I$, contained in the kernel of the natural map   $U_i*U_j \to \UU_A(\QQ)$. Suppose that the following conditions hold:
\begin{enumerate}
\item $\mathcal R$ contains $U_i^{(|a_{ij}|+1)}(U_j)$ and $U_j^{(|a_{ji}|+1)}(U_i)$; 
\item the group $\UU_{\mathcal R}(\QQ):=U_i*U_j/\llangle \mathcal R \rrangle$ is nilpotent.
\end{enumerate}
Then the canonical homomorphism $\varphi_{\mathcal R}\co\UU_{\mathcal R}(\QQ)\to \UU_A(\QQ)$ is an isomorphism.
\end{lemma}
\begin{proof}
Proposition~\ref{prop:GGCQscalarmult} provides a $\QQ$-Lie algebra structure on $\Lie(\UU_{\mathcal R}(\QQ))$. Moreover, as in the proof of Proposition~\ref{prop:KLiealgtemp}, we have a surjective $\QQ$-Lie algebra morphism $\n^+_{\QQ}\to\Lie(\UU_{\mathcal R}(\QQ))\to \Lie(\UU_A(\QQ))$ by (1), whose composition with the Lie ring morphism $\Lie(\UU_A(\QQ))\to \n^+_{\QQ}$ is the identity on $\n^+_{\QQ}$. Hence the map $\Lie(\UU_{\mathcal R}(\QQ))\to \Lie(\UU_A(\QQ))$ induced by $\varphi_{\mathcal R}$ is an isomorphism. This shows that $\ker\varphi_{\mathcal R}$ is contained in $\gamma_{\infty}(\UU_{\mathcal R}(\QQ))$. But $\gamma_{\infty}(\UU_{\mathcal R}(\QQ))=\{1\}$ by (2), and hence $\varphi_{\mathcal R}$ is also injective.
\end{proof}

\begin{example}\label{example:simplerpresentationrank2}
Let $i,j\in I$ with $i\neq j$. Then $U_{ij}=\UU_{A_{ij}}(\QQ)$ is isomorphic to $\UU_{\mathcal R}(\QQ)=U_i*U_j/\llangle \mathcal R \rrangle$ in the following cases.
\begin{enumerate}
\item
If $(a_{ij},a_{ji})=(0,0)$ and $\mathcal R=\{[U_i,U_j]\}$;
\item
If $(a_{ij},a_{ji})=(-1,-1)$ and 
$$\mathcal R=\{U_i^{(2)}(U_j), U_j^{(2)}(U_i)\} = \{[U_i,[U_i,U_j]], [U_j,[U_j,U_i]]\};$$
\item
If $(a_{ij},a_{ji})=(-1,-2)$ and 
$$\mathcal R=\{U_i^{(2)}(U_j), U_j^{(3)}(U_i), [U_i, U_j^{(2)}(U_i)]\}.$$
\end{enumerate}

\end{example}

We remark that there is no obvious way to obtain a similar presentation of $\UU_{A_{ij}}(\QQ)$ of type $G_2$, i.e. if $(a_{ij},a_{ji})=(-1,-3)$.   Indeed, that group satisfies the relations 
\begin{align*}
\mathcal R_0= \big\{ &U_i^{(3)}(U_j), [U_j, U_i^{(2)}(U_j)], \\
& U_i^{(2)}(U_j^{(2)}(U_i)),  [U_j, U_i, U_j^{(2)}(U_i)]\!], U_j^{(4)}(U_i), \\
&U_i^{(2)}(U_j^{(3)}(U_i)), [U_j, U_i, U_j^{(3)}(U_i)]\!]\big\}.
\end{align*}
The quotient $H = U_i * U_j/\llangle \mathcal R_0 \rrangle$ is nilpotent, and the Lie ring $\Lie(H)$ is a $\QQ$-Lie algebra (by Proposition~\ref{prop:GGCQscalarmult}) of nilpotency class~$5$ and dimension~$7$. 
However,   the equality 
$$U_i^{(2)}(U_j) = [U_i, U_j^{(3)}(U_i)],$$ 
which holds in the group $\UU_{A_{ij}}(\QQ)$, does not hold in $H$, and it is not clear whether one could express this equality as a consequence of any larger set of relators exclusively consisting of basic commutators. 

\begin{remark}
It follows from Example~\ref{example:simplerpresentationrank2} that for each $2$-spherical GCM $A$ of rank $n$ such that $|a_{ij}| \leq 2$ for all $i, j$, the KMS group $\UU_A(\QQ)$ has a presentation as the quotient of a free product of $n$ copies of $\QQ$ modulo a set of relators  exclusively consisting of basic commutators. In that sense, the subject of the present paper is related to \cite{BM14}. 
\end{remark}

We now prove an analogue of Propositions~\ref{prop:UKMSscalarmult} and \ref{prop:KLiealgtemp} for $\U^+_A(\QQ)$. 

\begin{prop}\label{prop:U1scalarmult}
Set $\Gamma_m:=\gamma_m(\U_A^+(\QQ))$ for all $m$.  Then the Lie ring 
$$\Lie(\U_A^+(\QQ))=\bigoplus_{n\geq 1}\Gamma_n/\Gamma_{n+1}$$
is a $\QQ$-Lie algebra with respect to the scalar multiplication defined by (\ref{eqn:defscalarmult}). \\
Moreover, the assignment $e_i\mapsto x_i\Gamma_2$ extends to a $\QQ$-Lie algebra isomorphism $$\n^+_{\QQ}\to \Lie(\U_A^+(\QQ)).$$
\end{prop}
\begin{proof}
Let $\pi\co \n^+_{\QQ}\to \Lie(\U_A^+(\QQ))$ denote the composition of the surjective Lie ring morphisms $\n^+_{\QQ}\to \Lie(\UU_A(\QQ))$ (provided by Proposition~\ref{prop:KLiealgtemp}) and $\Lie(\UU_A(\QQ))\to \Lie(\U_A^+(\QQ))$ (induced by the natural surjective group morphism $\UU_A(\QQ)\to\U^+_A(\QQ)$). Note that $\pi$ is injective: indeed, the composition of $\pi$ with the Lie ring morphism $\Lie(\U_A^+(\QQ))\to \n^+_{\QQ}$  (see (\ref{eqn:morphismton})) is the identity map on $\n^+_{\QQ}$. Hence $\pi$ is an isomorphism of Lie rings. 

This allows to transport the $\QQ$-Lie algebra structure from $\n^+_{\QQ}\cong \Lie(\UU_A(\QQ))$ to $\Lie(\U_A^+(\QQ))$: more formally, as in the proof of Proposition~\ref{prop:GGCQscalarmult}, it suffices to check that if $s\in\NN^*$ and $g\in\Gamma_n$ is such that $g^s\in\Gamma_{n+1}$, then $g\in\Gamma_{n+1}$. But if $x\in\n^+_{\QQ}$ is the preimage of $g\Gamma_{n+1}\in \Lie(\U_A^+(\QQ))$ under $\pi$, then by assumption $\pi(sx)=0$ and hence $sx=x=0$, that is, $g\in\Gamma_{n+1}$. 

In particular, since $\pi$ is $\QQ$-linear (the $\QQ$-module structures on $\n^+_{\QQ}\cong\Lie(\UU_A(\QQ))$ and $\Lie(\U_A^+(\QQ))$ are compatible by construction), it is a $\QQ$-Lie algebra isomorphism.
\end{proof}

\begin{corollary}\label{cor:Gamma_nUma_n}
$\gamma_n(\U^+_A(\QQ))=\U^+_A(\QQ)\cap \U^{ma}_n(\QQ)$ for all $n\in\NN$.
\end{corollary}
\begin{proof}
The inclusion $\subseteq$ is clear. Conversely, let $g\in \U^+_A(\QQ)\cap \U^{ma}_{n}(\QQ)$ for some $n\geq 2$, and suppose for a contradiction that $g\notin\gamma_n(\U^+_A(\QQ))$. Then $g\in \gamma_m(\U^+_A(\QQ))\setminus \gamma_{m+1}(\U^+_A(\QQ))$ for some $m\leq n-1$. Hence $g\gamma_{m+1}(\U^+_A(\QQ))$ is a nonzero element of $\Lie(\U^+_A(\QQ))$, and is thus mapped to a nonzero element in $\bigoplus_{n\geq 1}\U^{ma}_n(\QQ)/\U^{ma}_{n+1}(\QQ)\cong\n^+_{\QQ}$ under the morphism (\ref{eqn:morphismton}), since this morphism is injective by Proposition~\ref{prop:U1scalarmult}. In other words, $g\notin \U^{ma}_{m+1}(\QQ)$, and hence $g\notin \U^{ma}_{n}(\QQ)$, a contradiction.
\end{proof}

\begin{corollary}\label{corollary:proofCD}
Theorem~\ref{thmintro:main} and Corollary~\ref{corintro:isom_completions} hold. Moreover, the kernel of the natural map $\UU_A(\QQ)\twoheadrightarrow\U^+_A(\QQ)$ coincides with $\gamma_{\infty}(\UU_A(\QQ))$.
\end{corollary}
\begin{proof}
Theorem~\ref{thmintro:main} sums up Propositions~\ref{prop:KLiealgtemp} and \ref{prop:U1scalarmult}. 

Let us show next that the kernel of $\pi\co\UU_A(\QQ)\to\U^+_A(\QQ)$ coincides with $\gamma_{\infty}(\UU_A(\QQ))$. Assume that $g\in\UU_A(\QQ)\setminus\gamma_{\infty}(\UU_A(\QQ))$, and let $n\in\NN$ be such that $g\in \gamma_n(\UU_A(\QQ))\setminus\gamma_{n+1}(\UU_A(\QQ))$. Then $g\gamma_{n+1}(\UU_A(\QQ))$ is a nonzero element of $\Lie(\UU_A(\QQ))$, and hence $\pi(g)\gamma_{n+1}(\U^+_A(\QQ))$ is a nonzero element of $\Lie(\U^+_A(\QQ))$ by Theorem~\ref{thmintro:main}. In particular, $g\notin\ker\pi$. This shows that $\ker\pi\subseteq\gamma_{\infty}(\UU_A(\QQ))$, and the reverse inclusion follows from the fact that $\U^+_A(\QQ)$ is residually nilpotent (see Remark~\ref{rem:U+residually-nilp}).

Finally, we show that Corollary~\ref{corintro:isom_completions} holds. Since $\widehat{\U}^+_A(\QQ)$ and $\U^{ma+}_A(\QQ)$ are the Hausdorff completions of $\U^+_A(\QQ)$ with respect to the filtrations $(\gamma_n(\U^+_A(\QQ)))_{n\in\NN}$ and $(\U^+_A(\QQ)\cap \U^{ma}_n(\QQ))_{n\in\NN}$, the isomorphism $\widehat{\U}^+_A(\QQ)\to \U^{ma+}_A(\QQ)$ follows from Corollary~\ref{cor:Gamma_nUma_n}. On the other hand, we have just established that the natural map $\UU_A(\QQ)/\gamma_{\infty}(\UU_A(\QQ))\to\U^+_A(\QQ)$ is an isomorphism, and hence induces an isomorphism of the corresponding pro-nilpotent completions $\widehat{\UU}_A(\QQ)\to\widehat{\U}^+_A(\QQ)$.
\end{proof}

Note that Corollary~\ref{cor:Gamma_nUma_n} is no longer true without the $2$-sphericity assumption on $A$, as shown by the following lemma.
\begin{lemma}\label{lemma:nonormalformnon2sph}
Suppose that $A=\begin{psmallmatrix}2&a_{12}\\ a_{21}&2\end{psmallmatrix}$ with $a_{12}a_{21}\geq 4$. Then $\gamma_n(\U^+_A(\QQ))$ is properly contained in $\U^+_A(\QQ)\cap \U^{ma}_n(\QQ)$ for $n=|a_{12}|+3$.
\end{lemma}
\begin{proof}
By \cite[\S3(6)]{Morita88}, the real root groups $x_1(\QQ)$ and $x_2(\QQ)$ generate their free product in $\U^+_A(\QQ)$. Let $g:=[x_1,x_1,\dots,x_1,x_2]\!]\in \U^+_A(\QQ)$ ($|a_{12}|+1$ terms ``$x_1$''). Then $g\in \U^+_A(\QQ)\cap \U^{ma}_{N+1}(\QQ)$ where $N:=|a_{12}|+2$ because $(\ad e_1)^{|a_{12}|+1}e_2=0$ in $\n^+_{\QQ}$ (see \cite[Theorem~8.51(5)]{KMGbook}).

We claim that $g\notin \gamma_{N+1}(\U_A^+(\QQ))$. Indeed, write $\U_A^+(\QQ)=\bigcup_{n\geq 1}V_n$ where $V_n$ is the subgroup of $\U_A^+(\QQ)$ generated by $x_1(1/n)$ and $x_2(1/n)$ (hence a free group on these two elements). Suppose for a contradiction that $g\in \gamma_{N+1}(\U_A^+(\QQ))$. Then $g\in \gamma_{N+1}(V_n)$ for some large enough $n$. Hence $g\gamma_{N+1}(V_n)$ is zero in the Lie ring $\Lie(V_n)$. On the other hand, $\Lie(V_n)$ is the free Lie algebra on $y_1:=x_1(1/n)\gamma_{2}(V_n)$ and $y_2:=x_2(1/n)\gamma_{2}(V_n)$ by \cite[II \S5 n°4 Théorème 3]{BouLie}. Since
$$g\gamma_{N+1}(V_n)=n^N[y_1,y_1,\dots,y_1,y_2]\!]$$
($N-1$ terms ``$y_1$''), we get the desired contradiction.
\end{proof}

\begin{corollary}\label{cor:normalform}
Corollary~\ref{corintro:normalform} holds, but becomes false without the $2$-sphericity assumption on $A$.
\end{corollary}
\begin{proof}
The fact that every element $g$ of $\U^{ma+}_A(\QQ)$ has a unique expression of the form $$g=\prod_{x\in\BBB}u_x(\lambda_x)\in \U^{ma+}_A(\QQ)\quad\textrm{for some $\lambda_x\in\QQ$}$$ readily follows from \cite[Theorem~8.51]{KMGbook} (for any GCM $A$). 

If $A$ is $2$-spherical, then for each $n\in\NN$, the element $\prod_{n_x\geq n}u_x(\lambda_x)$ belongs to $\U^+_A(\QQ)\cap \U^{ma}_n(\QQ)= \gamma_n(\U^+_A(\QQ))$ by Corollary~\ref{cor:Gamma_nUma_n}, yielding Corollary~\ref{corintro:normalform}.

If $A$ is not assumed to be $2$-spherical, then Lemma~\ref{lemma:nonormalformnon2sph} shows that there exists an element $g\in (\U^+_A(\QQ)\cap\U^{ma}_n(\QQ))\setminus \gamma_n(\U^+_A(\QQ))$. The unique  expression $g=\prod_{x\in\BBB}u_x(\lambda_x)$ is then such that $g=\prod_{n_x\geq n}u_x(\lambda_x)\notin \gamma_n(\U^+_A(\QQ))$.
\end{proof}


\section{Characterisation of the residual nilpotence of $\UU_A(\QQ)$}

\subsection{Functoriality of $\UU_A(\KK)$}

Given two GCM $A=(a_{ij})_{i,j\in I}$ and $A'=(a'_{ij})_{i,j\in I'}$, we write $A\geq A'$ if $I'\subseteq I$ and $|a_{ij}|\geq |a'_{ij}|$ for all $i,j\in I'$.

\begin{lemma}\label{lemma:functorialityUUA}
Let $A=(a_{ij})_{i,j\in I}$ and $A'=(a'_{ij})_{i,j\in I'}$ be two $2$-spherical GCM with $A\geq A'$, and let $\KK$ be a ring. Then for $i,j\in I$ with $i\neq j$, $\gamma\in\Delta^+_{ij}$ and $a\in\KK$, the assignment 
 $$x_{\gamma}(a)\mapsto
    \begin{cases}
      x_{\gamma}(a) & \text{if $i,j\in I'$ and $\gamma\in\Delta(A')$}\\
      1 & \text{otherwise}
    \end{cases} $$
defines a surjective group morphism $\UU_A(\KK)\to\UU_{A'}(\KK)$.
\end{lemma}
\begin{proof}
Since this holds for $A$ of rank $2$ by \cite[\S8.5.3]{KMGbook}, it also holds in general by the definition of $\UU_A(\KK)$ and $\UU_{A'}(\KK)$.
\end{proof}

\begin{lemma}\label{lemma:injectivityJinI}
Let $J\subseteq I$ and let $\KK$ be a ring. Then the natural morphism $\UU_{A_J}(\KK)\to\UU_{A}(\KK)$ is injective.
\end{lemma}
\begin{proof}
The composition of $\UU_{A_J}(\KK)\to\UU_{A}(\KK)$ with the morphism $\UU_A(\KK)\to\UU_{A_J}(\KK)$ provided by Lemma~\ref{lemma:functorialityUUA} is the identity on $\UU_{A_J}(\KK)$, yielding the claim.
\end{proof}

\subsection{Distance between edges in $\Gamma(U_{ij}; U_i, U_j)$}

Recall that when $A$ is a spherical GCM of rank $2$, the elements of $\U^+_A(\KK)$ admit a normal form (\ref{eqn:normalformrank2}) (see \S\ref{subsection:SR2S}).

\begin{prop}\label{prop:computations_rank2}
Let $\KK$ be a ring, and let $A$ be a rank $2$ GCM, with simple roots $\alpha,\beta$. Let $n\in\NN$ and $a_1,\dots,a_n,b_1,\dots,b_n\in\KK$. Set
$$A_n:=\sum_{i=1}^na_i, \ B_n:=\sum_{i=1}^nb_i, \ R_n:=\sum_{i=1}^nb_iA_i, \ S_n:=\sum_{i=1}^nb_iA_i^2, \ T_n:=\sum_{i=1}^nb_iA_i^3$$
and
$$U_n:=\sum_{i=1}^nb_i^2A_i^3-\sum_{i=1}^nb_iA_i^3B_{i-1}+3\sum_{i=1}^nb_iA_i^2R_{i-1}$$
(with the usual convention that a sum over an empty set is zero). Set also $X_n:=x_{\beta}(b_n)x_{\alpha}(a_n)\dots x_{\beta}(b_1)x_{\alpha}(a_1)\in\U^+_A(\KK)$.
\begin{enumerate}
\item 
If $A$ is of type $A_2$, then
$$X_n=x_{\alpha}(A_n)x_{\beta}(B_n)x_{\alpha+\beta}(R_n).$$
\item 
If $A$ is of type $B_2$, with short root $\alpha$, then
$$X_n=x_{\alpha}(A_n)x_{\beta}(B_n)x_{\alpha+\beta}(R_n)x_{2\alpha+\beta}(S_n).$$
\item 
If $A$ is of type $G_2$, with short root $\alpha$, then
$$X_n=x_{\alpha}(A_n)x_{\beta}(B_n)x_{\alpha+\beta}(R_n)x_{2\alpha+\beta}(S_n)x_{3\alpha+\beta}(T_n)x_{3\alpha+2\beta}(U_n).$$
\end{enumerate}
\end{prop}
\begin{proof}
Note that (1) and (2) follow from (3) and the functoriality of $\U_A^+(\KK)=\UU_A(\KK)$ in $A$ (see Lemma~\ref{lemma:functorialityUUA}). Let us thus assume we are in case (3). The nontrivial commutation relations (\ref{eqn:commutationrel}) in $\U_A^+(\KK)$ between root group elements are then the following, for all $a,b\in\KK$:
\begin{align}
x_{\beta}(b)x_{\alpha}(a)&=x_{\alpha}(a)x_{\beta}(b)x_{\alpha+\beta}(ab)x_{2\alpha+\beta}(a^2b)x_{3\alpha+\beta}(a^3b)x_{3\alpha+2\beta}(a^3b^2)\label{eqn:R1}\\
x_{\alpha+\beta}(b)x_{\alpha}(a)&=x_{\alpha}(a)x_{\alpha+\beta}(b)x_{2\alpha+\beta}(2ab)x_{3\alpha+\beta}(3a^2b)x_{3\alpha+2\beta}(3ab^2)\label{eqn:R2}\\
x_{2\alpha+\beta}(b)x_{\alpha}(a)&=x_{\alpha}(a)x_{2\alpha+\beta}(b)x_{3\alpha+2\beta}(3ab)\label{eqn:R3}\\
x_{3\alpha+\beta}(b)x_{\beta}(a)&=x_{\beta}(a)x_{3\alpha+\beta}(b)x_{3\alpha+2\beta}(-ab)\label{eqn:R4}\\
x_{2\alpha+\beta}(b)x_{\alpha+\beta}(a)&=x_{\alpha+\beta}(a)x_{2\alpha+\beta}(b)x_{3\alpha+2\beta}(3ab).\label{eqn:R5}
\end{align}
In particular, (3) holds for $n=1$ by (\ref{eqn:R1}). Assume now that it holds for $n\geq 1$, and set $Z_n:=x_{2\alpha+\beta}(S_n)x_{3\alpha+\beta}(T_n)x_{3\alpha+2\beta}(U_n)$ and  $Y_n:=x_{\alpha+\beta}(R_n)Z_n$. Then
\begin{align*}
X_{n+1}&=x_{\beta}(b_{n+1})x_{\alpha}(a_{n+1})x_{\alpha}(A_n)x_{\beta}(B_n)Y_n\\
&=x_{\beta}(b_{n+1})x_{\alpha}(A_{n+1})x_{\beta}(B_n)Y_n\\
&=x_{\alpha}(A_{n+1})x_{\beta}(b_{n+1})x_{\alpha+\beta}(A_{n+1}b_{n+1})x_{2\alpha+\beta}(A_{n+1}^2b_{n+1})x_{3\alpha+\beta}(A_{n+1}^3b_{n+1})\\
&\qquad x_{3\alpha+2\beta}(A_{n+1}^3b_{n+1}^2)x_{\beta}(B_n)Y_n \quad \textrm{by (\ref{eqn:R1})}\\
&= x_{\alpha}(A_{n+1})x_{\beta}(b_{n+1})x_{\alpha+\beta}(A_{n+1}b_{n+1})x_{2\alpha+\beta}(A_{n+1}^2b_{n+1})x_{\beta}(B_n)\\
&\qquad x_{3\alpha+\beta}(A_{n+1}^3b_{n+1}) x_{3\alpha+2\beta}(A_{n+1}^3b_{n+1}^2-B_nA_{n+1}^3b_{n+1})Y_n \quad \textrm{by (\ref{eqn:R4})}\\
&= x_{\alpha}(A_{n+1})x_{\beta}(B_{n+1})x_{\alpha+\beta}(A_{n+1}b_{n+1})x_{2\alpha+\beta}(A_{n+1}^2b_{n+1})x_{3\alpha+\beta}(A_{n+1}^3b_{n+1})\\
&\qquad x_{3\alpha+2\beta}(A_{n+1}^3b_{n+1}^2-B_nA_{n+1}^3b_{n+1})x_{\alpha+\beta}(R_n)Z_n\\
&= x_{\alpha}(A_{n+1})x_{\beta}(B_{n+1})x_{\alpha+\beta}(R_{n+1})x_{2\alpha+\beta}(A_{n+1}^2b_{n+1})x_{3\alpha+\beta}(A_{n+1}^3b_{n+1})\\
&\qquad x_{3\alpha+2\beta}(A_{n+1}^3b_{n+1}^2-B_nA_{n+1}^3b_{n+1}+3R_nA_{n+1}^2b_{n+1})Z_n \quad \textrm{by (\ref{eqn:R5})}\\
&= x_{\alpha}(A_{n+1})x_{\beta}(B_{n+1})x_{\alpha+\beta}(R_{n+1})x_{2\alpha+\beta}(S_{n+1})x_{3\alpha+\beta}(T_{n+1})x_{3\alpha+2\beta}(U_{n+1}),
\end{align*}
as desired.
\end{proof}

\begin{lemma}\label{lemma:A2}
Let $\KK$ be a ring. Assume that $A$ is of type $A_2$, with simple roots $\alpha,\beta$. Let $g_1,g_2\in\U_A(\KK)$ be of the form
$$g_1=x_{\alpha}(a_2)x_{\beta}(b_1)x_{\alpha}(a_1)\quad\textrm{and}\quad g_2=x_{\beta}(b_2)x_{\alpha}(a_2)x_{\beta}(b_1)$$
for some $a_i,b_i\in\KK$. Let $i\in\{1,2\}$ and suppose $g_i=x_{\alpha+\beta}(r)$ for some $r\in\KK$. Then $r=0$.
\end{lemma}
\begin{proof}
For $i=1$, Proposition~\ref{prop:computations_rank2}(1) yields 
$$a_1+a_2=0, \quad b_1=0, \quad 0=r.$$
For $i=2$, Proposition~\ref{prop:computations_rank2}(1) yields
$$a_2=0, \quad b_1+b_2=0, \quad 0=r.$$
The lemma follows.
\end{proof}

\begin{lemma}\label{lemma:B2}
Let $\KK$ be a ring. Assume that $A$ is of type $B_2$, with short simple root $\alpha$ and long simple root $\beta$. Let $g_1,g_2\in\U_A(\KK)$ be of the form
$$g_1=x_{\beta}(b_2)x_{\alpha}(a_2)x_{\beta}(b_1)x_{\alpha}(a_1)\quad\textrm{and}\quad g_2=x_{\alpha}(a_3)x_{\beta}(b_2)x_{\alpha}(a_2)x_{\beta}(b_1)$$
for some $a_i,b_i\in\KK$. Let $i\in\{1,2\}$ and suppose $g_i=x_{\alpha+\beta}(r_1)x_{2\alpha+\beta}(r_2)$ for some $r_1,r_2\in\KK$. Then $r_1=0$ if and only if $r_2=0$.
\end{lemma}
\begin{proof}
For $i=1$, Proposition~\ref{prop:computations_rank2}(2) yields 
$$a_1+a_2=0, \quad b_1+b_2=0, \quad b_1a_1=r_1, \quad b_1a_1^2=r_2.$$
For $i=2$, Proposition~\ref{prop:computations_rank2}(2) yields
$$a_2+a_3=0, \quad b_1+b_2=0, \quad b_2a_2=r_1, \quad b_2a_2^2=r_2.$$
The lemma follows.
\end{proof}

\begin{lemma}\label{lemma:G2}
Let $\KK$ be a ring. Assume that $A$ is of type $G_2$, with short simple root $\alpha$ and long simple root $\beta$. Let $g_1,g_2\in\U_A(\KK)$ be of the form
$$g_1=x_{\beta}(b_3)x_{\alpha}(a_3)x_{\beta}(b_2)x_{\alpha}(a_2)x_{\beta}(b_1)x_{\alpha}(a_1)$$ and  $$g_2=x_{\alpha}(a_4)x_{\beta}(b_3)x_{\alpha}(a_3)x_{\beta}(b_2)x_{\alpha}(a_2)x_{\beta}(b_1)$$
for some $a_i,b_i\in\KK$. Let $i\in\{1,2\}$ and suppose $g_i=x_{\alpha+\beta}(r_1)x_{3\alpha+\beta}(r_2)x_{3\alpha+2\beta}(r_3)$ for some $r_1,r_2,r_3\in\KK$. Then $r_1=0$ if and only if $r_2=0$.
\end{lemma}
\begin{proof}
We focus on the case $i=1$, the case $i=2$ yielding the same equations (with $(a_1,a_2,a_3)$ replaced by $(a_2,a_3,a_4)$ and $(b_1,b_2,b_3)$ by $(b_2,b_3,b_1)$). Proposition~\ref{prop:computations_rank2}(3) yields the following equations:
$$a_1+a_2+a_3=0, \quad b_1+b_2+b_3=0, $$
and, using that $a_1+a_2=-a_3$,
$$b_1a_1-b_2a_3=r_1, \quad b_1a_1^2+b_2a_3^2=0, \quad b_1a_1^3-b_2a_3^3=r_2.$$
Multiplying by $a_3$ the first (resp. second) equation and adding the second (resp. third) yields
$$b_1a_1(a_1+a_3)=a_3r_1\quad\textrm{and}\quad b_1a_1^2(a_1+a_3)=r_2.$$
In particular, if $r_1=0$ then $r_2=0$. Conversely, if $r_2=0$, then either $r_1=0$ or $a_3=0$. But in the latter case, $b_1a_1=r_1$ and $b_1a_1^2=0$, so that $r_1=0$ as well, as desired.
\end{proof}

Recall from \S\ref{subsection:KMA} the definition of the integers $m_{ij}\in\NN$.
\begin{prop}\label{prop:distance-in-link}
Let $A=(a_{ij})_{i,j\in I}$ be a $2$-spherical GCM, and let $\KK$ be a ring. Let $i,j\in I$ with $i\neq j$ be such that $m_{ij}\in\{3,4,6\}$. 
Then, in the coset graph $\Gamma(U_{ij}; U_i, U_j)$, the distance between any two distinct elements of $U_{s_j(\alpha_i)}$,  seen as edges of the graph, is at least $m_{ij} +1$. 
\end{prop}
\begin{proof}
Recalling from \S\ref{subsection:CG} the characterisation of the distance between two edges of $\Gamma(U_{ij}; U_i, U_j)$, the proposition readily follows from Lemma~\ref{lemma:A2}, \ref{lemma:B2} or \ref{lemma:G2}, according to whether $m_{i,j}=3$, $4$, or $6$.
\end{proof}

\subsection{Girth of $\Gamma(U_{ij}; U_i, U_j)$}

In order to prove Theorem~\ref{thmintro:injiff3sph}, we will need the fact that the coset graphs $\Gamma(U_{ij}; U_i, U_j)$ have girth at least $2m_{ij}$. This is proved in \cite[Theorem~3.1 and Proposition~3.2]{LU93} (see also \cite[Proposition~7.1]{CCKW}). Since both these references actually assert that the girth of $\Gamma(U_{ij}; U_i, U_j)$ is \emph{equal} to $2m_{ij}$, although this is not true over $\FF_2$ in type $A_2$ and $G_2$, we chose to provide here the correct statement (and proof) of this result.

\begin{prop}\label{prop:girth}
Let $A$ be a GCM of type $A_2$, $B_2$ or $G_2$, with simple roots $\alpha,\beta$, and let $m_{\alpha\beta}=3$, $4$ or $6$ accordingly. Let $\KK$ be an integral domain.
\begin{enumerate}
\item
The girth of $\Gamma_{A,\KK}:=\Gamma(\U_A(\KK); \U_{\alpha}(\KK),\U_{\beta}(\KK))$ is equal to $2m_{\alpha\beta}$, unless $\KK=\FF_2$ and $A$ is of type $A_2$ or $G_2$.
\item
If $A$ is of type $A_2$ and $\KK=\FF_2$, then $\Gamma_{A,\KK}$ has girth $8$ (indeed $\Gamma_{A,\KK}$ is a cycle of length~$8$).
\item
If $A$ is of type $G_2$ and $\KK=\FF_2$, then $\Gamma_{A,\KK}$ has girth $16$ (indeed $\Gamma_{A,\KK}$ is the disjoint union of $4$ cycles of length~$16$).
\end{enumerate}
\end{prop}
\begin{proof}
(1) Let $k$ be the field of fractions of $\KK$. Then $\Gamma_{A,\KK}$ is a subgraph of $\Gamma_{A,k}:=\Gamma(\U_A(k); \U_{\alpha}(k),\U_{\beta}(k))$. Since $\Gamma_{A,k}$ has girth at least $2m_{\alpha\beta}$ (see \cite[Proposition~7.1]{CCKW}), the girth of $\Gamma_{A,\KK}$ is also at least $2m_{\alpha\beta}$. 

On the other hand, one readily checks using Proposition~\ref{prop:computations_rank2} the following identities for all $a,b\in\KK$.

If $A$ is of type $A_2$, then 
$$x_{\alpha}(a)x_{\beta}(b)x_{\alpha}(b-a)=x_{\beta}(b-a)x_{\alpha}(b)x_{\beta}(a).$$
If $\KK\neq\FF_2$, then there exist distinct $a,b\in\KK\setminus\{0\}$, yielding a cycle of length $6$.

If $A$ is of type $B_2$, then
$$x_{\alpha}(a)x_{\beta}(b)x_{\alpha}(-a)x_{\beta}(b)=x_{\beta}(b)x_{\alpha}(a)x_{\beta}(b)x_{\alpha}(-a).$$
Choosing $a=b=1$ yields a cycle of length $8$.

If $A$ is of type $G_2$, then
\begin{align*}
x_{\alpha}(a)x_{\beta}(b^3)x_{\alpha}(-a-b)x_{\beta}(a^3)&x_{\alpha}(b)x_{\beta}(-a^3-b^3)=\\
&x_{\beta}(-a^3-b^3)x_{\alpha}(a)x_{\beta}(b^3)x_{\alpha}(-a-b)x_{\beta}(a^3)x_{\alpha}(b).
\end{align*}
If $\charact\KK\neq 2$, choosing $a=b=1$ yields a cycle of length $12$. Assume now that $\charact \KK= 2$. If $|\KK|\geq 5$, then the polynomial $x^4-x=x(x-1)(x^2+x+1)$ is not identically zero over $\KK$. In this case, we obtain a cycle of length $12$ by choosing  $a\in\KK\setminus\{0,1\}$ such that $a^3\neq 1$ and $b=1$. Finally, if $\KK=\FF_4=\FF_2[y]/(y^2+y+1)$, then one readily checks using Proposition~\ref{prop:computations_rank2} that
$$x_{\alpha}(1)x_{\beta}(1)x_{\alpha}(y+1)x_{\beta}(1)x_{\alpha}(y)x_{\beta}(1)=x_{\beta}(1)x_{\alpha}(y)x_{\beta}(1)x_{\alpha}(y+1)x_{\beta}(1)x_{\alpha}(1),$$
also yielding a cycle of length $12$ in that case.

(2)(3) Note that, in the notations of Proposition~\ref{prop:computations_rank2}, if $X_n=1$ over $\FF_2$ with $a_n=b_{n-1}=a_{n-1}\dots=b_1=1$ and either $b_n=1$ or $a_1=1$, then the equations $A_n=B_n=0$ imply that $n$ is even and that $b_n=a_1=1$. Moreover, the equation $R_n=0$ implies that $n(n+1)/2=0$ and hence that $n$ is a multiple of $4$. Hence $\Gamma_{A,\KK}$ has girth at least $2n=8$ for $A$ of type $A_2$ and at least $2n=16$ for $A$ of type $G_2$, and one readily checks that indeed $X_n=1$ for these values of $n$. The more precise structure of  $\Gamma_{A,\KK}$ as a union of cycles follows from the fact that all its vertices have degree~$2$, and that the group $\U_{A}(\KK)$ (which has order $8$ or $64$ for $A$ of type $A_2$ or $G_2$) acts sharply transitively on the edges of  $\Gamma_{A,\KK}$.
\end{proof}

\subsection{Free products in $\UU_A(\KK)$}

\begin{theorem}\label{thm:freeproductKMS}
Let $A=(a_{ij})_{i,j\in I}$ be a $2$-spherical non-spherical GCM of rank $3$, and let $\KK$ be an integral domain. 
Choose the indexing set $I=\{1,2,3\}$ so that $a_{31},a_{32}\neq 0$. Then $\U_{s_3\alpha_1}(\KK)$ and $\U_{s_3\alpha_2}(\KK)$ generate their free product in $\UU_A(\KK)$.
\end{theorem}
 \begin{proof}
By definition, the KMS group $\UU_A(\KK)$ is the fundamental group of a triangle of groups (see \cite[Definition~II.12.12]{BH}): the face group is trivial, the edge groups are the simple root groups $U_i = x_i(\KK)$ with $i \in \{1, 2, 3\}$, and the vertex groups are the unipotent groups $U_{i j}$ with $i \neq j$ (the embeddings $U_i\hookrightarrow U_{ij}$ being the canonical ones). By \cite[Theorem~II.12.18]{BH}, the group $\UU_A(\KK)$ has a natural action on a triangle complex $Y$ that is sharply transitive on the triangles, and has three orbits of vertices. We fix a base triangle $C_0$, corresponding to the neutral element of $\UU_A(\KK)$, and
for each pair $\{i, j\} \subset \{1, 2, 3\}$, we denote by $v_{ij}$ the vertex of $C_0$ of which $U_{ij}$ is the stabiliser. For each $i\neq j$, the (simplicial) link of $v_{ij}$ is isomorphic (by construction of $Y$) to the bipartite coset graph $\Gamma(U_{ij}; U_i, U_j)$, with the edge opposite $v_{ij}$ in $C_0$ corresponding to the base edge $e_{ij}:=\{U_i,U_j\}$ in $\Gamma(U_{ij}; U_i, U_j)$. Moreover, metrising $C_0$ as a Euclidean or hyperbolic (hence CAT(0)) triangle with angle $\pi/m_{ij}$ at the vertex $v_{ij}$ yields a natural complete metric on $Y$ such that $\UU_A(\KK)$ acts on $Y$ by isometries (see \cite[Theorem~I.7.13]{BH}). This also turns $\Gamma(U_{ij}; U_i, U_j)$ into a metric graph by considering the geometric link $\mathrm{Lk}(v_{ij},Y)$ of $v_{ij}$ in $Y$ (see \cite[I.7.15]{BH}). Since loops in $\mathrm{Lk}(v_{ij},Y)$ have length at least $2m_{ij}.\tfrac{\pi}{m_{ij}}=2\pi$ by Proposition~\ref{prop:girth}, it moreover follows from \cite[Theorem~II.5.2 and Lemma~II.5.6]{BH} that the metric on $Y$ is CAT(0).

Since $a_{31},a_{32}\neq 0$ by assumption, we have $m_{13}, m_{23} \in \{3,4,6\}$.
It follows from Proposition~\ref{prop:distance-in-link} that for $i \in \{1, 2\}$, any two edges in the $\U_{s_i\alpha_3}(\KK)$-orbit of $e_{i3}$ in the graph $\Gamma(U_{i3}; U_i, U_3)$ are mutually at distance~$\geq m_{i3} +1$. 
This means that in the metric graph $\Gamma(U_{i3}; U_i, U_3)$, the midpoints of those edges are mutually at distance~$>\pi$ apart. Since the graph is bipartite and since $U_{i3}$ preserves that bipartition, it follows that if $v$ is a vertex belonging to $e_{i3}$, any two vertices in the $\U_{s_i\alpha_3}(\KK)$-orbit of $v$ are also mutually at distance~$> \pi$ apart. In particular, \cite[Theorem~I.7.16 and Remark~I.5.7]{BH} imply that for any two such vertices $v_1,v_2$, the geodesic segments $[v_{i3},v_1]$ and $[v_{i3},v_2]$ in $Y$ form an angle $\pi$. In other words, letting $\sigma = [v_{13}, v_{23}]$ denote the edge of $Y$ spanned by the vertices $v_{13}$ and $v_{23}$ (that is, the edge of $C_0$ stabilised by $U_3$), we see that for all $g,h \in \U_{s_i\alpha_3}(\KK)$ with $g\neq h$, the angle at $v_{i3}$ formed by the geodesic segments $g\sigma$ and $h\sigma$ is equal to $\pi$.

Let now $T$ denote the Bass--Serre tree of the free amalgamated product $\U_{s_1\alpha_3}(\KK) * \U_{s_2\alpha_3}(\KK)$, and let 
$$\varphi \colon \U_{s_1\alpha_3}(\KK) * \U_{s_2\alpha_3}(\KK)\to \langle \U_{s_1\alpha_3}(\KK) \cup  \U_{s_2\alpha_3}(\KK)\rangle \leq \UU_A(\KK)$$ 
be the natural homomorphism. We view $T$ as a metric tree, each of whose edges has length equal to the length of the geodesic segment $\sigma \subset Y$. Using that the free product  $ \U_{s_1\alpha_3}(\KK) * \U_{s_2\alpha_3}(\KK)$ acts sharply transitively on the edge set of $T$, we obtain a natural $\varphi$-equivariant map $f \colon T \to Y$, ranging in the $1$-skeleton of $Y$, and  mapping the base edge of $T$ to $\sigma$. The previous discussion ensures that $f$ is locally an isometric embedding. Since $Y$ is complete (see \cite[Theorem~I.7.13]{BH}) and simply connected (see \cite[Corollary~II.1.5]{BH}), it follows from \cite[Prop.~II.4.14]{BH} that $f$ is a global isometric embedding (in particular $f(T)$ is convex), and hence that $\varphi$ is injective. Thus the subgroup $\langle \U_{s_1\alpha_3}(\KK) \cup  \U_{s_2\alpha_3}(\KK)\rangle \leq \UU_A(\KK)$ is isomorphic to the free product $\U_{s_1\alpha_3}(\KK) * \U_{s_2\alpha_3}(\KK)$, as desired.
\end{proof}

\noindent
{\bf Proof of Theorem~\ref{thmintro:injiff3sph}}: 
In view of Lemma~\ref{lemma:injectivityJinI}, there is no loss of generality in assuming that $A$ has rank $3$. Assume that $A$ is not spherical, and let us show that $\UU_A(\KK)\to\U^+_A(\KK)$ is not injective. Up to reordering the indexing set $I=\{1,2,3\}$ so that $a_{31},a_{32}\neq 0$, we know from Theorem~\ref{thm:freeproductKMS} that $\U_{s_3\alpha_1}(\KK)$ and $\U_{s_3\alpha_2}(\KK)$ generate their free product in $\UU_A(\KK)$. On the other hand, by (\ref{eqn:siactiononrootgroups}) in \S\ref{subsection:MKMG}, $\U_{s_3\alpha_1}(\KK)$ and $\U_{s_3\alpha_2}(\KK)$ generate the group $\widetilde{s}_3\langle \U_{\alpha_1}(\KK),\U_{\alpha_2}(\KK)\rangle \widetilde{s}_3\inv\cong \U_{A_{12}}^+(\KK)$ in $\U_A^+(\KK)$, yielding the claim. \hspace{\fill} \qed

\noindent
{\bf Proof of Corollary~\ref{corintro:residnilp}}: 
The first statement is contained in Corollary~\ref{corollary:proofCD}. The equivalence of (i) and (ii) follows from Remark~\ref{rmk:res-nilp-IFF-tf-res-nilp}, and the equivalence of (i) and (iii) from the first statement and Theorem~\ref{thmintro:injiff3sph}. \hspace{\fill} \qed

\bibliographystyle{amsalpha} 
\bibliography{KMSLie}

\end{document}